\documentclass[10pt]{article}
\usepackage{enumitem}
\usepackage{amsfonts}
\usepackage{amsmath}
\usepackage{graphicx}
\usepackage{epstopdf}
\usepackage{fancyhdr}
\usepackage{amsthm}
\usepackage{cite}
\usepackage{geometry}
\usepackage{subfigure}
\usepackage[figuresright]{rotating}
\usepackage{caption}
\usepackage{bm}
\usepackage{amssymb}
\usepackage{amscd}
\usepackage{float}
\usepackage{color}
\usepackage{tabularx}
\usepackage{diagbox}
\usepackage{tabularx}
\usepackage{mathtools}
\usepackage{multicol}
\usepackage{multirow}
\usepackage{makecell}
\usepackage[justification=centering]{caption}
\usepackage{rotfloat}
\usepackage{longtable}
\allowdisplaybreaks[4]

\newtheorem{thm}{Theorem}[section]
\newtheorem{cor}{Corollary}[section]
\newtheorem{lemma}{Lemma}[section]

\newtheorem{assumption}{Assumption}[section]
\newtheorem{remark}{Remark}[section]

\newtheorem{defn}{Definition}[section]

\newtheorem{note}{Notations}[section]
\newcounter{nextauthor}
\def\mathrm{\mbox}
\begin{document}
\title{{\bf Stability for a stochastic fractional differential variational inequality with L\'{e}vy jump}\thanks{This work was supported by the National Natural Science Foundation of China (12171339, 12471296).}}
\author{Yue Zeng$^a$, Yao-jia Zhang$^{b,c}$ and Nan-jing Huang$^a$ \thanks{Corresponding author. E-mail addresses: njhuang@scu.edu.cn; nanjinghuang@hotmail.com} \\
{\small\it  a. Department of Mathematics, Sichuan University, Chengdu, Sichuan 610064, P.R. China}\\
{\small\it b. School of Science, Southwest Petroleum University, Chengdu, 610500, P.R. China}\\
{\small\it c. Institute for Artificial Intelligence, Southwest Petroleum University, Chengdu, 610500, P.R. China}}
\date{}
\maketitle \vspace*{-9mm}
\begin{abstract}
\noindent
The main goal of this paper is to investigate the multi-parameter stability result for a stochastic fractional differential variational inequality with L\'{e}vy jump (SFDVI with L\'{e}vy jump) under some mild conditions. We  verify that Mosco convergence of the perturbed set implies point convergence of the projection onto the Hilbert space consisting of special stochastic processes whose range is the perturbed set. Moreover, by using the projection method and some inequality techniques, we establish a strong convergence result for the solution of SFDVI with L\'{e}vy jump when the mappings and constraint set are both perturbed. Finally, we apply the stability results to the spatial price equilibrium problem and the multi-agent optimization problem in stochastic environments.
\end{abstract}
\section{Introduction}
 In 2008, Pang and Stewart \cite{Pang2008} conducted a comprehensive study of differential variational inequalities (DVIs) in finite dimensional spaces. To deal with uncertainties in dynamical systems, Zhang et al. \cite{zhang2023} has recently extended this framework by introducing the following new class of stochastic differential inequalities (SDVI) consisting of stochastic differential equations and stochastic variational inequalities:
\begin{equation}\label{S1}
\left\{
\begin{aligned}
&dx(t)=f(t,x(t),u(t))dt+g(t,x(t),u(t))dB_t, \, t\in [0,T], \, x(0)=x_0,\\
&\left\langle F(t,\omega,x(t,\omega),u(t,\omega)),v-u(t,\omega) \right\rangle \geq 0, \; \forall v\in K, \; a.e. \, t\in [0,T], \; a.s. \, \omega\in \Omega,
\end{aligned}
\right.
\end{equation}
where $B_t$ is an $l$-dimensional standard Brownian motion, $x_0$ is a given random variable, $K$ is a closed convex subset in $\mathbb{R}^q$, $f,g,F$ are proper measurable functions. They established the existence and uniqueness of the solution for \eqref{S1}, as well as the parameter dependency of the solution, and they have also provided a series of application examples. On the basis of this work, they also developed a penalty method for solving SDVI \cite{zhang2023r} and formulated its Euler scheme \cite{zhang2023s}.  It is well known that complementarity problems (CPs) represent a significant class of nonlinear optimization problems with broad applications in economics, engineering, and other fields of applied mathematics \cite{Facchinei2003}, and varational inequalities is deeply related to complementarity theory. Similar to the relationship between classical variational inequality and classical nonlinear complementarity problem \cite{Facchinei2003}, if $K$ is a closed convex cone, then \eqref{S1} is equivalent to the following stochastic differential complementarity problem:
\begin{equation}\label{CP1}
\left\{
\begin{aligned}
&dx(t)=f(t,x(t),u(t))dt+g(t,x(t),u(t))dB_t, \, t\in [0,T], \, x(0)=x_0,\\
&K\ni u(t,\omega)\; \bot\;  F(t,\omega,x(t,\omega),u(t,\omega))\in K^*, \; a.e. \, t\in [0,T], \; a.s. \, \omega\in\Omega,
\end{aligned}
\right.
\end{equation}
where the notation $\bot$ means "perpendicular" and  $K^*=\{d\in\mathbb{R}^q:\langle v,d\rangle\geq 0,\;\forall v\in K\}$ is the dual cone of the closed convex cone $K$.

It is worth mentioning that numerous systems, used to model the problems arising in the real world, exhibit properties of memory and jumps. To accurately describe such phenomena, some scholars have incorporated fractional calculus and L\'{e}vy jumps into their models to capture the systems' memory and jump properties, rather than relying solely on Brownian motion to characterize the system's behavior (see, for example, \cite{Abouagwa2019,Kumar2017,Pedjeu2012,Xu2015}). Considering the impact of memory and jumps on practical systems, Zeng et al. \cite{zeng2024} investigated  the following stochastic fractional differential variational inequality with L\'{e}vy jump (SFDVI with L\'{e}vy jump):
\begin{align}\label{S2}
\begin{cases}
 dx(t)=b(t,x(t-),u(t-))dt+\sigma_1(t,x(t-),u(t-))(dt)^\alpha+\sigma(t,x(t-),u(t-))dB(t)\\
 \qquad\qquad\mbox{} +\int_{\| y\|<c}G(t,x(t-),u(t-),y)\tilde{N}(dt,dy), \; \alpha \in (\frac{1}{2},1),\\
 x(0)=p_0,\\
 \langle F(t,\omega,x(t,\omega),u(t,\omega)),v-u(t,\omega)\rangle\geq0, \; \forall v\in K, \; a.e. \, t\in [0,T], \; a.s. \, \omega\in\Omega,
 \end{cases}
\end{align}
which is composed of a stochastic fractional differential equation with L\'{e}vy jump and a stochastic variational inequality. For the fractional differential part, they considered a special class of easy to calculate fractional differentiation defined by $(dt)^{\alpha}$ with $\alpha\in(\frac{1}{2},1)$ (see \cite{Pedjeu2012,zeng2024}). Under some mild conditions, they showed that there exists a unique solution to system \eqref{S2}. For more symbols and details, please refer to \cite{zeng2024}. Clearly, when $K$ is a closed convex cone, SFDVI with L\'{e}vy jump is equivalent to the following stochastic fractional differential complementarity problem with L\'{e}vy jump:
\begin{align}\label{CP2}
  \begin{cases}
 dx(t)=b(t,x(t-),u(t-))dt+\sigma_1(t,x(t-),u(t-))(dt)^\alpha+\sigma(t,x(t-),u(t-))dB(t)\\
 \qquad\qquad\mbox{} +\int_{\| y\|<c}G(t,x(t-),u(t-),y)\tilde{N}(dt,dy), \; \alpha \in (\frac{1}{2},1),\\
 x(0)=p_0,\\
  K\ni u(t,\omega)\; \bot\;  F(t,\omega,x(t,\omega),u(t,\omega))\in K^*, \; a.e. \, t\in [0,T], \; a.s. \, \omega\in\Omega.
 \end{cases}
  \end{align}

During the past decades, many scholars have produced very excellent results on the properties and numerical analysis related to the solutions of DVI under different conditions (see, for example, \cite{Chen2014,Li2010,Li2015,Liu2017,Migorski2022,Zeng2018}). As a class of stochastic versions of DVIs, SDVI \eqref{S1} and SFDVI with L\'{e}vy jump \eqref{S2} provide effective mathematical models for solving many practical problems, such as circuit problems with diodes and bridge collapse problems in stochastic environments \cite{zhang2023s,zhang2023,zhang2023r,zeng2024}. Therefore, as a further development of DVI, it would be interesting and important to study the properties associated with the solutions of SDVI \eqref{S1} and SFDVI with L\'{e}vy jump \eqref{S2}.

On the other hand, in the modern science and engineering fields, once the existence of a solution for a given system is established, the study of the system's stability becomes crucial. This is because it is related to the behavior of the system in the face of external perturbations. For unconstrained dynamical systems and their stochastic counterparts such as ordinary differential equation and stochastic differential equations, stability typically refers to Lyapunov stability, which aims to study the system's long-term behavior under external perturbations, for instance we refer to \cite{Khasminskii2012,Yin2011,Min2023,Shen2022} and the references therein. Indeed, DVI can be regarded as a constrained dynamical system, while \eqref{S1} and \eqref{S2} can be seen as the constrained stochastic dynamical systems. However, different from the Lyapunov sense, the study of the stability of solutions to variational inequalities typically refers to the change in the set of solutions to variational inequalities when the constraint set or mapping is perturbed. In the context of DVI, research on the stability of the solution set is generally divided into two approaches. One is to study the continuity or semi-continuity about the solution set of DVI under parameter perturbation. For example, Pang et al. \cite{Pang2009} in 2009 studied the dependence of the solution to DVI on initial values, Wang et al. \cite{Wang2013} in 2013 established some semi-continuous results for a class of differential vector variational inequalities and Guo et al. \cite{Guo2020} in 2020 discussed the upper semi-continuity and continuity of the set of mild solutions to partial differential variational inequalities in Banach spaces with respect to two parameters. The other approach is to study the convergence of the solution set of DVI under parametric perturbation. For example, by using the concept of
the Mosco convergence for set sequences, Gwinner \cite{Gwinner2013} in 2013 obtained a novel upper set convergence result to the solutions of a new class of DVIs with respect to perturbations in the data, Sofonea \cite{Sofonea2018} in 2018 studied a convergence result for a class of elliptic hemivariational inequalities which describes the dependence of the solution with respect to the data, and  Xiao et al. \cite{Xiao2022} in 2022 obtained the strong convergence to the unique solution to the evolutionary variational-hemivariational inequality under different mild conditions. For more stability results for DVIs, the readers are encouraged to consult \cite{Wu2024,Li2019,Jiang2021,Wang2014,Tang2020} and the references therein. As extensions of DVI within a stochastic framework, systems \eqref{S1} and \eqref{S2}  exhibit more complex behavior due to the introduction of stochastic elements. However, to the best of our knowledge, stability analysis of \eqref{S1} and \eqref{S2} have not been studied in the literature, apart from  \cite{zhang2023} wherein, as mentioned above, a special consideration of the continuity dependence of solutions of SDVI with respect to the parameters in the mappings. Therefore, as mentioned above, as two important and meaningful extensions of DVI in the direction of stochastic analysis, it is necessary and attractive to study the stability of \eqref{S1} and \eqref{S2} under some mild conditions.

The present paper is thus devoted to the multi-parameter stability for the solution of SFDVI with L\'{e}vy jump. More precisely, we would like to consider the stability for the following multi-parameter system $({\bf MPS}(\lambda,\mu))$ associated with the system \eqref{S2}:
\begin{align}\label{S5}
\begin{cases}
 dx_{\lambda,\mu}(t)=b_{\lambda}(t,x_{\lambda,\mu}(t-),u_{\lambda,\mu}(t-))dt
 +\sigma_{1\lambda}(t,x_{\lambda,\mu}(t-),u_{\lambda,\mu}(t-))(dt)^\alpha\\
 \qquad\qquad\;\mbox{}+\sigma_{\lambda}(t,x(t-),u(t-))dB(t)+\int_{\| y\|<c}G_{\lambda}(t,x_{\lambda,\mu}(t-),u_{\lambda,\mu}(t-),y)\tilde{N}(dt,dy), \; \alpha \in (\frac{1}{2},1),\\
 x_{\lambda,\mu}(0)=p_0,\\
 \langle F_{\lambda}(t,\omega,x_{\lambda,\mu}(t,\omega),u_{\lambda,\mu}(t,\omega)),v-u_{\lambda,\mu}(t,\omega)\rangle\geq0, \; \forall v\in K_{\mu}, \; a.e. \, t\in [0,T], \; a.s. \, \omega\in\Omega,
 \end{cases}
\end{align}
where $\lambda,\mu$ are two parameters in a complete metric space $(M,d)$,  and the mappings are perturbed by the parameter $\lambda$, the constraint set $K_{\mu}$ is a closed convex subset in $\mathbb{R}^q$, and $K_{\mu}$ is perturbed by the parameter $\mu$, with the solution denoted as the stochastic process pair $(x_{\lambda,\mu}(t,\omega),u_{\lambda,\mu}(t,\omega))$. It should be noted that a stochastic process such as $M(t,\omega)$ is a measurable function defined on the product space $[0,\infty]\times\Omega$, and $M(t,\cdot)$ is a random variable for each $t$, $M(\cdot,\omega)$ is a measurable function for each $\omega$ (called sample path). Without causing confusion, we will use $x_{\lambda,\mu}(t)$ and $u_{\lambda,\mu}(t)$ instead of $x_{\lambda,\mu}(t,\omega)$ and $u_{\lambda,\mu}(t,\omega)$ in the following statements. We assume that there exist two sequences $\{\lambda_m\}_{m\in\mathbb{Z}_+}$, $\{\mu_n\}_{n\in\mathbb{Z}_+}$ which belong to the metric space $(M,d)$ such that $\lim_{m\rightarrow\infty}d(\lambda_m,\lambda)\rightarrow0$, $\lim_{n\rightarrow\infty}d(\mu_n,\mu)\rightarrow0$ and the multi-parameter perturbed systems are denoted by ${\bf MPS}(\lambda_m,\mu_n)$. More symbolic explanations and details will be explained later.

Unlike previous studies, the solutions to SFDVI with L\'{e}vy jump are essentially stochastic processes. Such stochastic processes tend to discuss the relevant properties of the solutions in a complete probability space with filtration, implying that we must consider solutions that hold in the sense of $``\forall v\in K,\;a.e. \, t\in [0,T], \; a.s. \, \omega\in\Omega"$ rather than $``\forall v\in K"$ as in some classical variational inequalities and our solution is adapted to the filtration. To address this, we lift the SFDVI with L\'{e}vy jump from a finite-dimensional space
$K$ to an equivalent problem in a Hilbert space. Such an approach effectively solves the problem of the existence uniqueness of the solution of SFDVI with L\'{e}vy jump \cite{zeng2024}, but poses a new challenge to the study of stability with parametric perturbations, i.e., assessing the impact of perturbations to the constraint set
$K$ in a finite-dimensional space on the equivalence problem in an infinite-dimensional Hilbert space. To the best of our knowledge, such issues have been only partially explored in Gwinner's static random variational inequalities in \cite{Gwinner2000} and dynamic non-random differential variational inequalities in \cite{Gwinner2013}. However, in a dynamic stochastic environment, the measurability and integrability that arise in such problems are more complicated. In present paper, we overcome this difficulty and further extend such interesting results to general stochastic differential variational inequalities with dynamics. On the other hand, the interaction of $x_{\lambda,\mu}(t)$ and $u_{\lambda,\mu}(t)$ in the face of a perturbation of the two parameters also poses a challenge to our stability results. And clearly, the methods used in those aforementioned papers do not apply here in a straightforward manner. Indeed,  we need to carefully find reasonable assumptions for the perturbed set $K_{\mu}$ while assuming that the mappings have continuity with respect to the parameter $\lambda$, which ensures the stability of the solution of ${\bf MPS}(\lambda,\mu)$. Our main contributions are as follows:
\begin{description}
  \item[(i)]  extending results in \cite{Gwinner2000,Gwinner2013} to general stochastic differential variational inequalities, and showing the pointwise convergence of the projection onto the Hilbert space containing some special stochastic processes whose range is the perturbed set $K_{\mu}$ by using the Mosco convergence of the perturbed set $K_{\mu}$ with respect to the parameter $\mu$ (more details will be given in Section 3).
  \item[(ii)] proposing a new multi-parameter stability result for system \eqref{S5} when the mappings and constraint set are both perturbed by employing the projection method and some inequality techniques, more precisely, we establish a strong convergence result for the solution of SFDVI with L\'{e}vy jump.
  \item[(iii)] applying the theoretical results to the spatial price equilibrium problem in stochastic environments and showing the stability result for such a problem under some mild conditions.
  \item[(iv)] investigating a new class of multi-agent stochastic fractional differential games with L\'{e}vy jump, whose Nash equilibrium can be characterized by the SFDVI with L\'{e}vy jump, and then providing the existence, uniqueness, and stability of the Nash equilibrium for such a problem by employing our theoretical results.
\end{description}

The rest of this paper is organised as follows. The next section recalls some relevant symbols, definitions and known results. After that in Section 3, we introduce Mosco convergence and give the relationship between Mosco convergence and projection. In Section 4, we propose and prove a stability result of the solution of \eqref{S5} under some mild  conditions. In Section 5, we give two applications to the stochastic spatial price equilibrium problem and the stochastic multi-agent optimization problem, before we summarize the results in Section 6.

\section{Preliminaries}
\setcounter{equation}{0}

In this section we give the definitions of SFDVI with L\'{e}vy jump, after recalling some basics of stochastic analysis and the definition of fractional calculus.

\subsection{Basics of  stochastic analysis and fractional calculus}
Let $(\Omega, \mathcal{F}, \{\mathcal{F}_t\}_{t\geq0}, \mathbb{P})$ is a complete probability space with filtration $\{\mathcal{F}_t\}_{t\geq0}$. The filtration $\{\mathcal{F}_t\}_{t\geq0}$ is an increasing family $\{\mathcal{F}_t:t\geq 0\}$ of $\sigma$-fields. A stochastic process $M(t)$ is said to be adapted to $\{\mathcal{F}_t:t\geq 0\}$ is for each $t$, the random variable $M(t,\cdot)$ is $\mathcal{F}_t$-measurable. Moreover, an adapted stochastic process $M(t)$ is called a martingale with respect to $\{\mathcal{F}_t:t\geq 0\}$ if for any $s<t$,
$$
\mathbb{E}(M(t)|\mathcal{F}_s)=M(s)
.$$

Now, we fill in the details in \eqref{S2} and \eqref{S5}.
\begin{note}\label{funs}
For any $\lambda,\mu\in(M,d)$, assume that the measurable and adapted functions appeared in \eqref{S5} satisfy the following conditions.
 \begin{itemize}
  \item $\|\cdot\|$ and $\langle\cdot,\cdot\rangle$ are the norm and the inner product in $\mathbb{R}^p$ (or $\mathbb{R}^q$), respectively.
 \item $B(t)$ is an $l$-dimensional $\mathcal{F}_t$-adapted Brownian motion.
\item $N:\mathbb{R}^+\times (\mathbb{R}^p\backslash\{0\})$ is independent of $B(t)$ and is an $\mathcal{F}_t$-adapted Poisson measure, and the associated compensated martingale measure is defined by $$\tilde{N}(dt,dy):=N(dt,dy)-v(dy)dt,$$
       where $v(\cdot)$ is the intensity measure satisfying $$\int_{\mathbb{R}^p\backslash\{0\}}\frac{y^2}{1+y^2}v(dy)<\infty.$$

   \item $\int_0^s\int_{\| y\|<c}G_{\lambda}(t,x_{\lambda,\mu}(t-),u_{\lambda,\mu}(t-),y)\tilde{N}(dt,dy)$ is an $\mathbb{R}^p$-valued square integrable martingale satisfying
       $$P\left(\int_0^s\int_{\| y\|<c}\|G_{\lambda}(t,x_{\lambda,\mu}(t-),u_{\lambda,\mu}(t-),y)\|^2 v(dy)dt<\infty\right)=1$$
       and the maximum allowable jump size is defined as constant $c\geq0$.
   \item $p_0$ is the initial value satisfying $\mathbb{E}\| p_0\|^2<\infty$.
   \item $b_{\lambda}:[0,T]\times\mathbb{R}^p\times\mathbb{R}^q\rightarrow \mathbb{R}^p$.
   \item $ \sigma_{\lambda}:[0,T]\times\mathbb{R}^p\times\mathbb{R}^q\rightarrow \mathbb{R}^{p\times l}$.
   \item $G_{\lambda}:[0,T]\times\mathbb{R}^p\times\mathbb{R}^q\times \mathbb{R}^p\rightarrow \mathbb{R}^p$.
   \item $F_{\lambda}:[0,T]\times\Omega\times\mathbb{R}^p\times K\rightarrow\mathbb{R}^q$.
   \item $\sigma_{1\lambda}:[0,T]\times\mathbb{R}^p\times\mathbb{R}^q\rightarrow \mathbb{R}^p$ is a continuous function with respect to $t$.
 \end{itemize}
\end{note}

The following notations are also used in this paper
\begin{note}\label{notation}
\noindent
\begin{itemize}
  \item $\mathcal{L}^2(\Omega,\mathbb{R}^p)$ is a Hilbert space which contains all $\mathbb{R}^p$-valued square integrable random variables and is equipped with a norm defined by $\|\cdot\|_{\mathcal{L}^2}=(\mathbb{E}\|\cdot\|^2)^{\frac{1}{2}}$.

  \item $\mathbb{Z}_+=\{1,2,3,...\}.$
  \item The projection of $v$ onto $S$,  denoted by $P_S(v)$.
 \item $H[a,b]= \mathcal{L}^2_{ad}([a,b]\times\Omega,\mathbb{R}^q)$ is a Hilbert space which contains all $\mathbb{R}^q$-valued $\mathcal{F}_t$-adapted stochastic processes satisfying $\int_a^b\mathbb{E}\| f(t,\omega)\|^2dt<\infty$ for all $f\in H[a,b]$ and is equipped with an inner product defined by
      $$
      \langle u,v\rangle_{H[a,b]}=\int_a^b\mathbb{E}(\langle u(t,\omega),v(t,\omega)\rangle)dt,\; \forall u,v\in H[a,b],\;[a,b]\subset[0,T].
      $$
  \item  For any given $\mu\in(M,d)$, let $$U_{\mu}[a,b]=\left\{u(t,\omega)\in\mathcal{L}^2_{ad}([a,b]\times\Omega,\mathbb{R}^q): u(t,\omega)\in K_{\mu}, \;a.e.\,t\in[a,b],\; a.s.\,\omega\in\Omega\right\},$$
       where  $K_{\mu}$ is a closed convex subset in $\mathbb{R}^q$ with $0\in K_{\mu}$.
  \item For any given sequence $\{\mu_n\}\subset (M,d)$  with $d(\mu_n,\mu)\to 0$, let
  $$U_{\mu_n}[a,b]=\left\{u(t,\omega)\in\mathcal{L}^2_{ad}([a,b]\times\Omega,\mathbb{R}^q): u(t,\omega)\in K_{\mu_n}, \;a.e.\,t\in[a,b],\; a.s.\,\omega\in\Omega\right\},$$
  where perturbed set $K_{\mu_n}$ is a closed convex subset in $\mathbb{R}^q$.
\end{itemize}
\end{note}

Next we recall some basic definitions and results of stochastic analysis and fractional calculus.
\begin{defn}\label{RLFI}\cite{Pedjeu2012} For any given $g\in \mathcal{L}^1([a, b]; \mathbb{R}^d)$,  the left Riemann-Liouville fractional integrals of order $\alpha$ is defined by
 $$(I_{a+}^{\alpha}g)(t)=\frac{1}{\Gamma(\alpha)}\int_a^t(t-s)^{\alpha-1}g(s)ds, \; t>a,$$
 where $\Gamma(\alpha)=\int_0^\infty s^{\alpha-1}e^{-s}ds$.
\end{defn}
\begin{defn}\label{RLFD}\cite{Biagini2008} Let $g\in \mathcal{L}^1([a, b]; \mathbb{R}^d)$ and $\alpha\in(0,1)$. If $g$ is absolutely continuous on $[a,b]$, then the left Riemann-Liouville fractional derivatives of order $\alpha$ is defined by
$$(D_{a+}^{\alpha}g)(t)=\frac{1}{\Gamma(1-\alpha)}\frac{d}{dt}\int_a^t(t-s)^{-\alpha}g(s)ds, \;t>a.$$
\end{defn}
\begin{remark}\label{REF}
In this paper, we will only consider the scenario with $a=0$ in Definition \ref{RLFD}, that is,
$$(D_{0+}^{\alpha}g)(t)=\frac{1}{\Gamma(1-\alpha)}\frac{d}{dt}\int_0^t(t-s)^{-\alpha}g(s)ds.$$
According to \cite{Jumarie2004}, one has $D_{0+}^{\alpha}=\frac{d^{\alpha}}{(dt)^{\alpha}}$ and $(d^\alpha
g)(t)=\Gamma(1+\alpha)(dg)(t)=(D_{0+}^\alpha g)(t)(dt)^\alpha$. Let $f(t)=(D_{0+}^\alpha g)(t)$. Then the following formula is valid:
$$\int_0^tf(s)(ds)^\alpha=\Gamma(1+\alpha)g(t)=\Gamma(1+\alpha)D_{0+}^{-\alpha} f(t)=\frac{\Gamma(1+\alpha)}{\Gamma(\alpha)}\int_0^t(t-s)^{\alpha-1}f(s)ds=\alpha\int_0^t(t-s)^{\alpha-1}f(s)ds.$$
For further information, interested readers may refer to \cite{Abouagwa2019,Pedjeu2012,Jumarie2004}.
\end{remark}
\begin{lemma}\label{rdoob}\cite{Yong1999}(Doob's Inequality)
For $p\geq1$, assume that $x(t)$ is a right-continuous martingale such that $\mathbb{E}\|x(t)\|^p<\infty$ for $t\geq0$. Then
$$\mathbb{P}\left(\sup_{t\in[0,T]}\| x(t)\|>\lambda\right)\leq \frac{\mathbb{E}\| x(T)\|^p}{\lambda^p}, \; \forall T>0$$
and for $p>1$,
$$\mathbb{E}\left(\sup_{t\in[0,T]}\|x(t)\|^p\right)\leq \left(\frac{p}{p-1}\right)^p\mathbb{E}\| x(T)\|^p,  \; \forall T>0.$$
\end{lemma}
\begin{lemma}\label{ito} \cite{Oksendal}(It\^{o}'s Isometry) For any given $T>0$, one has
 $$\mathbb{E}\left[\left(\int^T_0f(t,\omega)dB_t\right)^2\right]=\mathbb{E}\left[\int_0^Tf^2(t,\omega)dt\right],\;\forall f\in\mathcal{V}(0,T),$$
 where $\mathcal{V}(0,T)$ is the set of all functions $f:[0,T]\times \Omega\rightarrow\mathbb{R}$ satisfying the following conditions:
 \begin{itemize}
   \item[(i)] $f$ is $\mathcal{B}\times\mathcal{F}$ measurable, where $\mathcal{B}$ denotes the Borel $\sigma$-algebra on $[0,T]$;
   \item[(ii)]  $f$ is $\mathcal{F}_t$-adapted;
   \item[(iii)] $\mathbb{E}\left[\int_0^Tf^2(t,\omega)dt\right]< \infty$.
 \end{itemize}
\end{lemma}

\begin{remark}\label{levyito}
In fact, according to \cite{Applebaum2004,Nunno2009}, when $G_{\lambda}(t,x_{\lambda,\mu},u_{\lambda,\mu},x)$ and $N(t,x)$ are defined as described in Notations \ref{funs} and  \ref{notation}, there is a similar It\^{o} isometry for the pure L\'{e}vy jump process $$\int_0^T\int_{\| y\|<c}G_{\lambda}(s,x_{\lambda,\mu}(s-),u_{\lambda,\mu}(s-),y)\tilde{N}(ds,dy)$$ as follows:
\begin{align*}
\quad&\mathbb{E}\left[\left(\int_0^T\int_{\| y\|<c}G_{\lambda}(s,x_{\lambda,\mu}(s-),u_{\lambda,\mu}(s-),y)\tilde{N}(ds,dy)\right)^2\right]\\
=&\mathbb{E}\left[\int_0^T\int_{\| y\|<c}G_{\lambda}^2(s,x_{\lambda,\mu}(s-),u_{\lambda,\mu}(s-),y)v(dy)ds\right]
\end{align*}
For more information on the L\'{e}vy process, the readers may wish to refer to \cite{Applebaum2004,Nunno2009,Abouagwa2019,Xu2015}.
\end{remark}

\subsection{SFDVI with L\'{e}vy jump}
In this subsection, we give the definition of the solutions of SFDVI with L\'{e}vy jump and we also give some lemmas and theorems that lead to our main results.

\begin{defn}\label{SOL} A pair $(x(t),u(t))$ is said to be a solution to SFDVI with L\'{e}vy jump if and only if $x\in \mathcal{L}^2_{ad}([0,T]\times \Omega,\mathbb{R}^p)$ satisfying
\begin{align}\label{3}
 \begin{cases}
 dx(t)=b(t,x(t-),u(t-))dt+\sigma_1(t,x(t-),u(t-))(dt)^\alpha+\sigma(t,x(t-),u(t-))dB(t)\\
 \qquad\qquad+\int_{\| y\|<c}G(t,x(t-),u(t-),y)\tilde{N}(dt,dy), \; \alpha \in (\frac{1}{2},1),\\
 x(0)=x_0,\\
u(t)\in SOL(U_0[0,T],F(t,\omega,x(t,\omega),u(t,\omega))),
 \end{cases}
\end{align}
where $U_0[0,T]=\left\{u(t,\omega)\in\mathcal{L}^2_{ad}([0,T]\times\Omega,\mathbb{R}^q): u(t,\omega)\in K, \;a.e.\,t\in[0,T],\; a.s.\,\omega\in\Omega\right\}$, $K$ is a closed convex subset in $\mathbb{R}^q$, and $SOL(U_0[0,T],F(t,\omega,x(t,\omega),u(t,\omega)))$ is the set of solutions of the stochastic variational inequality:  find $u\in U_0[0,T]$ such that
\begin{equation}\label{2}
  \langle F(t,\omega,x(t,\omega),u(t,\omega)),v-u(t,\omega)\rangle\geq0, \; \forall v\in K, \; a.e. \, t\in [0,T], \; a.s. \, \omega\in\Omega.
\end{equation}
If the solution $(x(t),u(t))$  is unique in the sense of almost everywhere, we say it is the unique solution to system SFDVI with L\'{e}vy jump.
\end{defn}
According to Remark \ref{REF}, the first equation in system \eqref{3} can be rewritten as
\begin{align*}
  x(t)&= x_0+\int_0^t b(s,x(s-),u(s-))ds+\alpha\int_0^t(t-s)^{\alpha-1}\sigma_1(s,x(s-),u(s-))ds+\int_0^t\sigma(s,x(s-),u(s-))dB(s)\\
  &\quad +\int_0^t\int_{\| y\|<c}G(s,x(s-),u(s-),y)\tilde{N}(ds,dy),\; \alpha \in (\frac{1}{2},1).
\end{align*}

\begin{lemma}\label{rconvex}\cite{zhang2023}
 For any $\mu\in (M,d)$, if $K_{\mu}$ be a non-empty, closed and convex subset of $\mathbb{R}^q$. Then for any $[a,b]\subset [0,T]$, $U_{\mu}[a,b]$ is a non-empty, closed and convex subset of $\mathcal{L}^2_{ad}([a,b]\times\Omega,\mathbb{R}^q)$.
\end{lemma}

\begin{lemma}\label{requi}\cite{zhang2023}
For any fixed $\lambda,\mu\in(M,d)$ and given $x_{\lambda,\mu}\in \mathcal{L}^2_{ad}([a,b]\times \Omega,\mathbb{R}^p)$, if $u_{\lambda,\mu}\in U_{\mu}[a,b]$, then the following
$$
\langle F_{\lambda}(t,\omega,x_{\lambda,\mu}(t,\omega),u_{\lambda,\mu}(t,\omega)),v-u_{\lambda,\mu}(t,\omega)\rangle\geq0,\;\forall v\in K_{\mu},\;a.e.\,t\in[a,b],\;a.s.\,\omega\in\Omega
$$
is equivalent to the following variational inequality
$$
\langle\tilde{F}_{\lambda}(x_{\lambda,\mu},u_{\lambda,\mu}),v'-u_{\lambda,\mu}\rangle_{H_{[a,b]}}\geq0,\; \forall v' \in U_{\mu}[a,b],
$$
where $\tilde{F}_{\lambda}:\mathcal{L}^2_{ad}([a,b]\times \Omega,\mathbb{R}^p)\times U_{\mu}[a,b]\rightarrow  \mathcal{L}_{ad}^2([a,b]\times\Omega,\mathbb{R}^q)$ is defined by
\begin{align*}
&\tilde{F}_{\lambda}(x_{\lambda,\mu},u_{\lambda,\mu})(s,\omega):=F_{\lambda}(s,\omega,x_{\lambda,\mu}(s,\omega),u_{\lambda,\mu}(s,\omega)),\\ &\forall (x_{\lambda,\mu},u_{\lambda,\mu})\in \mathcal{L}^2_{ad}([a,b]\times \Omega,\mathbb{R}^p)\times U_{\mu}[a,b],\; \forall s\in[a,b],\; \forall \omega\in\Omega.
\end{align*}
\end{lemma}

\section{Mosco convergence of sets and strong convergence of projections}
\setcounter{equation}{0}
In this section, we extend the results in \cite{Gwinner2000,Gwinner2013} to general stochastic differential variational inequalities, and furthermore, we show that the Mosco convergence of the perturbed set implies point convergence of the projection onto the Hilbert space consisting of special stochastic processes whose range is the perturbed set.

\begin{defn}\label{mosco}\cite{Mosco1969}(Mosco convergence)
Let $\{C_n\}_{n\in \mathbb{Z}_{+}}$ be a sequence of closed convex subsets of a Hilbert space $H$, the sequence $\{C_n\}_{n\in \mathbb{Z}_{+}}$ is called Mosco convergent to a closed convex subset $C_0$ of $H$, written $C_n\overset{M}{\rightarrow}C_0$, if and only if
$$
 w-Ls_n C_n \subset C_0 \subset s-Li_n C_n.
$$
Here $s-Li_n C_n$ and $w-Ls_n C_n$ are defined as follows:
\begin{itemize}
  \item $x\in s-Li_n C_n$ if and only if there exists a sequence $\{x_n\}\subset C_n$ such that $\{x_n\}$ converges strongly to $x$;
  \item $x\in w-Ls_n C_n$ if and only if there exists a subsequence $\{C_{n_i}\}$ of $\{C_n\}$ and $\{x_{n_i}\}\subset C_{n_i}$ such that $\{x_{n_i}\}$ converges weakly to $x$.
\end{itemize}
\end{defn}

\begin{defn}\label{banach1}\cite{wang2011}
A Banach space $E$ is said to
\begin{description}
  \item[(i)] be strictly convex if for all $x,y\in S_E=\{z\in E: \|z\|=1\}$ with $x\neq y$ $\Rightarrow \|x+y\|<2$;
  \item[(ii)]  be uniformly convex if for any $\epsilon\in(0,2]$, there exists $\delta>0$ such that $$\forall x,y\in S_E,\; \|x-y\|\geq\epsilon\Rightarrow \left\|\frac{x+y}{2}\right\|<1-\delta;$$
  \item[(iii)] have the Kadec-Klee property if for any sequence $\{x_n\}_{n\in \mathbb{Z_{+}}}$ in E with $x_n\rightharpoonup x_0\in E(weak\; convergence)$ and $\lim_{n\rightarrow\infty}\|x_n\|=\|x_0\|$, we have $x_n\rightarrow x_0(strong\; convergence)$;
  \item[(iv)] be smooth if the norm of E is G\^{a}teaux differentiable, and the norm of E is G\^{a}teaux differentiable if $\lim_{t\rightarrow 0}h(x,y,t)$ exists for any $x,y\in S_E$, where $h:S_E\times S_E\times \mathbb{R}\backslash\{0\}\rightarrow\mathbb{R}$ and
      $h(x,y,t)=\frac{\|x+ty\|-\|x\|}{t}.$
\end{description}
\end{defn}

\begin{lemma}\label{banach}\cite{wang2011,Ibaraki2003}
Let $E$ be a smooth, reflexive, and strictly convex Banach space having the Kadec-Klee property. Assume that $\{C_n\}_{n\in\mathbb{Z}_{+}}$ and $C_0$ are all  nonempty closed convex subsets of $E$. If $C_n\overset{M}{\rightarrow}C_0$, then $P_{C_n}(x)$ converges strongly to $P_{C_0}(x)$ for all $x\in E$.
\end{lemma}

\begin{lemma}\label{proconv}
If $C_n\overset{M}{\rightarrow}C$, and $C_n,\;n\in \mathbb{Z}_+$, $C$ are all nonempty closed convex subsets of $H[0,T]$, then $P_{C_n}(u)$ converges strongly to $P_{C}(u)$ for all $u\in H[0,T]$.
\end{lemma}
\begin{proof}
It is enough to verify that $H[0,T]$ is a smooth, reflexive, and strictly convex Banach space having the Kadec-Klee property by Lemma \ref{banach}. Because $H[0,T]$ is a Hilbert space, it is obvious that $H[0,T]$ is a reflexive smooth Banach space (see \cite{Ammar2023}). Moreover, for any $u,v\in H[0,T]$,  one has the following formula
\begin{equation}\label{pro1}
\|u+v\|_{H[0,T]}^2+\|u-v\|_{H[0,T]}^2=2(\|u\|_{H[0,T]}^2+\|v\|_{H[0,T]}^2).
\end{equation}
From Definition \ref{banach1}, for any $u,v\in S_{H[0,T]}$ satisfying $\|u-v\|_{H[0,T]}\geq\epsilon$ with $\epsilon\in(0,2]$, it follows from \eqref{pro1} that $$\left\|\frac{u+v}{2}\right\|^2_{H[0,T]}<1-\left\|\frac{u-v}{2}\right\|^2_{H[0,T]}\leq 1-\frac{{\epsilon}^2}{4}.$$
This shows that $H[0,T]$ is uniformly convex Banach space. It is well known that every uniformly convex Banach space is strictly convex and enjoys the Kadec-Klee property (see Chapter 8 of \cite{Ammar2023}) and so $H[0,T]$ is a smooth, reflexive, and strictly convex Banach space having the Kadec-Klee property.
\end{proof}

\begin{lemma}\label{K00}\cite{Bauschke2017}(Polar set) Let $C$ be a subset of a Hilbert space $H$. The polar set of $C$ is
$$
C^o=\{y\in H|\;\langle x,y\rangle\leq1,\forall x\in C\}.
$$
Moreover, $C$ is closed, convex, and contains the origin if and only if $C^{oo}=C$.

\end{lemma}

\begin{lemma}\cite{Gwinner2013}\label{gwinner}
For any $n\in\mathbb{Z}_{+}$, let $K_n$ and $K$ are all closed convex subsets of a Hilbert space $V$. If $K_{n}\overset{M}{\rightarrow}K$, then $K_{n}^o\overset{M}{\rightarrow}K^o$.
\end{lemma}

\begin{lemma}\label{Un}
Let $K_{\mu}$ and $K_{\mu_n}$ be defined by Notations \ref{notation}. If $K_{\mu_n}\overset{M}{\rightarrow}K_{\mu}$,
 then $U_{\mu_n}[0,T]\overset{M}{\rightarrow}U_{\mu}[0,T].$
\end{lemma}
\begin{proof}
We first claim: 1. $w-Ls_n U_{\mu_n}[0,T] \subset U_{\mu}[0,T]$.

For any $f\in w-Ls_n U_{\mu_n}[0,T]$, there exists a subsequence $\{U_{\mu_{n_i}}[0,T]\}$ of $U_{\mu_n}[0,T]$ and $f_{i}\in H[0,T]$ such that $\{f_i\}$ converges weakly to $f$ and $f_i\in U_{\mu_{n_i}}[0,T]$. By Lemma \ref{K00} and Notations \ref{notation}, one has
$K_{\mu}=K_{\mu}^{oo},$ and so it is enough to show that $f(t,\omega)\in K_{\mu}^{oo},\;a.e. \, t\in [0,T], \; a.s. \, \omega\in\Omega$, which means that for any $\xi\in K_{\mu}^o$, for $a.e. \, t\in [0,T], \; a.s. \, \omega\in\Omega$ there holds $\langle f(t,\omega),\xi \rangle\leq1$. Assume the contrapositive, then there exists $\bar{\xi}\in K_{\mu}^o$ and $A=\{(t,\omega)|\; \langle f(t,\omega),\bar{\xi}\rangle>1,t\in[0,T],\;\omega\in\Omega\}$ such that the measure $|A|=\int^T_0\int_{\Omega}1_{A}(t,\omega)d\mathbb{P}dt>0$, where
$$1_{A}(t,\omega)=\begin{cases}
 1, & (t,w)\in A; \\
 0,  & (t,\omega)\notin A.
\end{cases}$$
Let $\bar{u}(t,\omega)=\frac{1}{|A|}\bar{\xi}1_{A}(t,\omega)$. By Lemma \ref{gwinner}, one has
$K_{\mu_n}^o\overset{M}{\rightarrow}K_{\mu}^o.$ Thus there exists a sequence $\bar{\xi}_i\in K^o_{\mu_{n_i}}$ such that $\bar{\xi_i}$ converges strongly to $\bar{\xi}\in K_{\mu}^o$. Letting $\bar{u}_i(t,\omega)=\frac{1}{|A|}\bar{\xi}_i1_{A}(t,\omega)$, one has $\bar{u}_i\rightarrow \bar{u}$ in $H[0,T]$. By construction,
$$
\langle f_i,\bar{u}_i\rangle=\frac{1}{|A|}\int^T_0\int_{\Omega}\langle f_i(t,\omega),\bar{\xi}_i1_{A}(t,\omega)\rangle d\mathbb{P}dt\leq1.
$$
Thus in the limit we arrive at
$$
1\geq \langle f,\bar{u}\rangle=\frac{1}{|A|}\int^T_0\int_{\Omega}\langle f(t,\omega),\bar{\xi}1_{A}(t,\omega)\rangle d\mathbb{P}dt>1,
$$
which is a contradiction. This means that the claim 1 is true.

Next we claim: 2. $U_{\mu}[0,T] \subset s-Li_n U_{\mu_n}[0,T]$.

For any $f\in U_{\mu}[0,T]$, we need to find a sequence $\{f_n\}$ with $f_n\in U_{\mu_n}[0,T]$ and $f_n\rightarrow f$ in $H[0,T]$. Let $f_n(t,\omega)=P_{K_{\mu_n}}(f(t,\omega)),\;t\in[0,T],\;\omega\in\Omega.$ We only need to verify that, for each $n\in \mathbb{Z}_+$, $f_n$ is $\mathcal{F}_t$-adapted, integrable and $f_n\rightarrow f$ in $H[0,T]$.

(i) We verify that $f_n$ is $\mathcal{F}_t$-adapted. In fact, it is well known that $P_{K_{\mu_n}}:\mathbb{R}^q\rightarrow\mathbb{R}^q$ is continuous and so $P_{K_{\mu_n}}$ is Borel measurable. For any $B\in\mathcal{B}(\mathbb{R}^q)$, $P^{-1}_{K_{\mu_n}}(B)=C\in\mathcal{B}(\mathbb{R}^q)$. For any fixed $t\in[0,T]$, let $g_t(\omega)=f(t,\omega)$. Then it is clear that $g_t^{-1}(C)\in \mathcal{F}_t$ since $f\in U_{\mu}[0,T]$. Thus for any fixed $t\in[0,T]$, $f^{-1}_n(B)=g_t^{-1}P^{-1}_{K_{\mu_n}}(B)\in\mathcal{F}_t$, which means $f_n$ is $\mathcal{F}_t$-adapted.

(ii) We show that $f_n$ is integrable in $H[0,T]$. Indeed,  for any $n\in\mathbb{Z}_+$, choose $c_n\in K_{\mu_n}$, one has $P_{K_{\mu_n}}(c_n)=c_n$. Moreover,
 \begin{align*}
 \int_0^T\mathbb{E}\|f_n(t,\omega)\|^2dt
 &=\int_0^T\mathbb{E}\|P_{K_{\mu_n}}(f(t,\omega))-P_{K_{\mu_n}}(c_n)+P_{K_{\mu_n}}(c_n)\|^2dt\nonumber\\
 &\leq 2\int_0^T\mathbb{E}\|P_{K_{\mu_n}}(f(t,\omega))-P_{K_{\mu_n}}(c_n)\|^2dt+2\int_0^T\mathbb{E}\|P_{K_{\mu_n}}(c_n)\|^2dt\nonumber\\
 &\leq2\int_0^T\mathbb{E}\|f(t,\omega)-c_n\|^2dt+2T\|c_n\|^2\nonumber\\
 &\leq4\int_0^T\mathbb{E}\|f(t,\omega)\|^2dt+4\int_0^T\mathbb{E}\|c_n\|^2dt+2T\|c_n\|^2\nonumber\\
 &\leq4\int_0^T\mathbb{E}\|f(t,\omega)\|^2dt+6T\|c_n\|^2\nonumber\\
 &<\infty.
 \end{align*}
Thus $f_n$ is integrable in $H[0,T]$ and so $f_n\in U_{\mu_n}[0,T]$.

 (iii) We prove that $f_n\rightarrow f$ in $H[0,T]$. In fact, by Lemma \ref{banach}, one has $f_n(t,\omega)=P_{K_{\mu_n}}(f(t,\omega))\rightarrow P_{K_{\mu}}(f(t,\omega))=f(t,\omega)$ for all $t\in[0,T]$ and  $\omega\in\Omega$. Thus, for any $c\in K_{\mu}$, there exists a sequence $\{c_n\}$ such that $c_n\in K_{\mu_n}$ and $c_n\rightarrow c$ in $\mathbb{R}^q$. For any $t\in[0,T]$ and $\omega\in\Omega$,
\begin{align*}
\|f_n(t,\omega)\|&\leq\|f_n(t,\omega)-c_n\|+\|c_n\|\nonumber\\
&\leq\|P_{K_{\mu_n}}(f(t,\omega))-P_{K_{\mu_n}}(c_n)\|+\|c_n\|\nonumber\\
&\leq\|f(t,\omega)-c_n\|+\|c_n\|\nonumber\\
&\leq\|f(t,\omega)\|+2\|c_n\|.
\end{align*}
If $g_n(t,\omega)=\|f(t,\omega)\|+2\|c_n\|$ and $g=\|f(t,\omega)\|+2\|c\|$, then $\|f_n(t,\omega)\|\leq g_n(t,\omega)$ for all $t\in[0,T]$ and $\omega\in\Omega$, and
\begin{align}\label{gn}
\int_0^T\mathbb{E}\|g_n(t,\omega)-g(t,\omega)\|^2dt&\leq4\int_0^T\mathbb{E}\|\|c_n\|-\|c\|\|^2dt\nonumber\\
&\leq 4\int_0^T\mathbb{E}\|c_n-c\|^2dt\nonumber\\
&\leq 4T\|c_n-c\|^2.
\end{align}
It follows from \eqref{gn} and $c_n\rightarrow c$ that $g_n\rightarrow g$ in $H[0,T]$. By the dominated convergence theorem, we know that $f_n\rightarrow f$ in $H[0,T]$.

Up to now, by (i)-(iii), there is a sequence $f_n=P_{K_{\mu_n}}(f(t,\omega))\in U_{\mu_n}[0,T]$ such that $f_n\rightarrow f$ in $H[0,T]$. Thus, the claim 2 is true.

In conclusion, it follows from claims 1 and 2 that $U_{\mu_n}[0,T]\overset{M}{\rightarrow}U_{\mu}[0,T]$.
\end{proof}
\begin{remark}\label{G2013}
Stochastic processes possess more complex properties of convergence, measurability and integrability. Consequently, some methods for functional analysis of deterministic Lebesgue space cannot be directly applied. In order to extend the results of Lemma 2 in \cite{Gwinner2013} from the deterministic case to the stochastic one, it is necessary to employ a different method from that used in \cite{Gwinner2013}, as outlined in Lemma \ref{Un}. In addition, an  analogous result for a very special static random variational inequality can be consulted in \cite{Gwinner2000}. However, while Gwinner in \cite{Gwinner2000} only gave a convergence result for a special constraint set consisting of measurable functions and closed convex cones for the static case, our result yields Mosco convergence for more general closed convex sets in a more complex dynamic stochastic environment.
\end{remark}

Now we have the following new result for the Mosco convergence and projections.
\begin{thm}\label{pro}
 If $K_{\mu_n}\overset{M}{\rightarrow}K_{\mu}$, then for any $u\in H[0,T]$, $P_{U_{\mu_n}[0,T]}(u)$ converges strongly to $P_{U_{\mu}[0,T]}(u)$.
\end{thm}
\begin{proof}
It follows from Lemma \ref{Un} that $U_{\mu_n}[0,T]\overset{M}{\rightarrow}U_{\mu}[0,T].$ Using Lemmas \ref{rconvex} and \ref{proconv}, we can see directly that, for any $u\in H[0,T]$, $P_{U_{\mu_n}[0,T]}(u)$ converges strongly to $P_{U_{\mu}[0,T]}(u)$.
\end{proof}

\section{Stability of the solutions to SFDVI with L\'{e}vy jump}
\setcounter{equation}{0}

In this section, we are in the position to consider the multi-parameter stability for SFDVI with L\'{e}vy jump. To this end, admit perturbations $b_{\lambda_m},\sigma_{\lambda_m},\sigma_{1\lambda_m},G_{\lambda_m},F_{\lambda_m}$ of the mappings $b_{\lambda},\sigma_{\lambda},\sigma_{1\lambda},G_{\lambda},F_{\lambda}$, and $K_{\mu_n}$ of the convex closed subset $K_{\mu}\subset \mathbb{R}^q$. Similar to the study of the stability for DVIs in \cite{Gwinner2013}, we obtain  some conditions on $b_{\lambda_m},\sigma_{\lambda_m},\sigma_{1\lambda_m},G_{\lambda_m},F_{\lambda_m}$ and $K_{\mu_n}$ for ensuring the solution sequence of perturbed systems ${\bf MPS}(\lambda_m,\mu_n)$ converges to a solution for system ${\bf MPS}(\lambda,\mu)$. In the sequel, we impose the following assumptions:
\begin{assumption}\label{ass}
  For any $\lambda,\mu\in(M,d)$, and any $t\in [0,T]$ with $T>0$, $x,x_1,x_2\in \mathbb{R}^p,u,u_1,u_2\in \mathbb{R}^q$, and $\tilde{x},\tilde{x}_1,\tilde{x}_2\in \mathcal{L}^2_{ad}([0,T]\times \Omega,\mathbb{R}^p), \tilde{u}_1,\tilde{u}_2\in H[0,T]$, assume that there exist some positive constants $\bar{C}$, $L_{b_{\lambda}}$, $L_{\sigma_{1\lambda}}$, $L_{\sigma_{\lambda}}$, $L_{g}$, $L_{G_{\lambda}}$, $K_{b_{\lambda}}$, $K_{\sigma_{1\lambda}}$, $K_{\sigma_{\lambda}}$, $K_{G_{\lambda}}$, and $L_{F}$ with $L_{F}>\bar{C}$ such that

\begin{itemize}
  \item[(i)]
  \noindent

   $\| b_{\lambda}(t,x,u) \|^2\leq K_{b_{\lambda}}(1+\| x \|^2+\| u \|^2)$;

  $\| \sigma_{\lambda}(t,x,u) \|^2_{\mathbb{R}^{n\times l}}\leq K_{\sigma_{\lambda}}(1+\| x \|^2+\| u \|^2)$;

   $\| \sigma_{1\lambda}(t,x,u) \|^2\leq K_{\sigma_{1\lambda}}(1+\| x \|^2+\| u \|^2)$;

   $\int_{\| y\|<c}\| G_{\lambda}(t,x,u,y) \|^2v(dy)\leq K_{G_{\lambda}}(1+\| x \|^2+\| u \|^2)$;

  \item[(ii)]
 \noindent

  $\| b_{\lambda}(t,x_1,u_1)-b_{\lambda}(t,x_2,u_2) \|^2\leq L_{b_{\lambda}}(\| x_1-x_2 \|^2+\| u_1-u_2 \|^2)$;

  $\| \sigma_{\lambda}(t,x_1,u_1)-\sigma_{\lambda}(t,x_2,u_2) \|^2_{\mathbb{R}^{n\times l}}\leq L_{\sigma_{\lambda}}(\| x_1-x_2 \|^2+\| u_1-u_2 \|^2)$;

  $\| \sigma_{1\lambda}(t,x_1,u_1)-\sigma_{1\lambda}(t,x_2,u_2) \|^2\leq L_{\sigma_{1\lambda}}(\| x_1-x_2 \|^2+\| u_1-u_2 \|^2)$;

  $\int_{\| y\|<c}\| G_{\lambda}(t,x_1,u_1,y)-G_{\lambda}(t,x_2,u_2,y) \|^2v(dy)\leq L_{G_{\lambda}}(\| x_1-x_2 \|^2+\| u_1-u_2 \|^2)$;
  \item[(iii)] $\tilde{F}_{\lambda}$ defined by Lemma \ref{requi} satisfies
  \begin{align*}
  &\| \tilde{F}_{\lambda}(\tilde{x}_1,\tilde{u}_1)-\tilde{F}_{\lambda}(\tilde{x}_2,\tilde{u}_2)\|_{H[0,T]}\leq L_{F}(\| \tilde{x}_1-\tilde{x}_2 \|_{H[0,T]}+\| \tilde{u}_1-\tilde{u}_2 \|_{H[0,T]});\\
&\langle \tilde{F}_{\lambda}(\tilde{x},\tilde{u}_1)-\tilde{F}_{\lambda}(\tilde{x},\tilde{u}_2),\tilde{u}_1-\tilde{u}_2 \rangle_{H[0,T]} \geq \bar{C}\| \tilde{u}_1-\tilde{u}_2\|_{H[0,T]};
\end{align*}

  \item[(iv)]
    \noindent
    $b_{\lambda},\sigma_{\lambda},\sigma_{1\lambda},G_{\lambda},F_{\lambda}$ are all uniform continuous with respect to $\lambda$.
\end{itemize}
\end{assumption}

Thanks to the proof of Lemmas 3.1 in \cite{zhang2023s}, we have the following lemma.

\begin{lemma}\cite{zhang2023s}\label{r31}
For any fixed $\mu\in(M,d)$ and fixed $x \in \mathcal{L}^2_{ad}([0,T]\times \Omega,\mathbb{R}^p)$, if $\tilde{F}_{\lambda}$ defined by Lemma \ref{requi} satisfies the condition $(iii)$ in Assumption \ref{ass}, then there exists a unique $u=P_{U_{\mu}[0,T]}(u-\rho \tilde{F}_{\lambda}(x,u))$ with $\rho>0$ such that
$$
\langle F_{\lambda}(t,\omega,x(t,\omega),u(t,\omega)),v-u(t,\omega)\rangle\geq0,\;\forall v\in K_{\mu},\;a.e.\,t\in[0,T],\;a.s.\,\omega\in\Omega.
$$
\end{lemma}

\begin{thm}\label{th1} \cite{zeng2024}(Existence and uniqueness) If conditions (i)-(iii) of Assumption \ref{ass} are satisfied, then for any $\lambda,\mu\in(M,d)$, the following system ${\bf MPS}(\lambda,\mu)$
\begin{align}\label{S3}
 \begin{cases}
 dx_{\lambda,\mu}(t)=b_{\lambda}(t,x_{\lambda,\mu}(t-),u_{\lambda,\mu}(t-))dt+\sigma_{1\lambda}(t,x_{\lambda,\mu}(t-),u_{\lambda,\mu}(t-))(dt)^\alpha\\
 \qquad\qquad\;\mbox{}+\sigma_{\lambda}(t,x_{\lambda,\mu}(t-),u_{\lambda,\mu}(t-))dB(t)\\
 \qquad\qquad\;\mbox{} +\int_{\| y\|<c}G_{\lambda}(t,x_{\lambda,\mu}(t-),u_{\lambda,\mu}(t-),y)\tilde{N}(dt,dy), \; \alpha \in (\frac{1}{2},1),\\
 x_{\lambda,\mu}(0)=p_0,\\
 \langle F_{\lambda}(t,\omega,x_{\lambda,\mu}(t,\omega),u_{\lambda,\mu}(t,\omega)),v-u_{\lambda,\mu}(t,\omega)\rangle\geq0, \; \forall v\in K_{\mu}, \; a.e. \, t\in [0,T], \; a.s. \, \omega\in\Omega
 \end{cases}
\end{align}
admits a unique solution $(x_{\lambda,\mu}(t),u_{\lambda,\mu}(t))\in \mathcal{L}^2_{ad}([0,T]\times \Omega,\mathbb{R}^p)\times U_{\mu}[0,T]$.
\end{thm}

Theorem \ref{th1} shows the existence and uniqueness of the solution of the system ${\bf MPS}(\lambda,\mu)$. Now we are in the position to present the multi-parameter stability result for ${\bf MPS}(\lambda,\mu)$.

\begin{thm}\label{multi}(Multi-parameter stability) Assume $\lambda_m, \mu_n\in(M,d)$ such that $\lambda_m\rightarrow\lambda$ and $\mu_n\rightarrow\mu$. Suppose  $b_{\lambda_m},b_{\lambda};\sigma_{\lambda_m},\sigma_{\lambda};\sigma_{1\lambda_m},\sigma_{1\lambda};G_{\lambda_m}, G_{\lambda}; F_{\lambda_m}, F_{\lambda};\tilde{F}_{\lambda_m},\tilde{F}_{\lambda_m}$ satisfy Assumption \ref{ass} and $K_{\mu_n}\overset{M}{\rightarrow}K_{\mu}$. Let $(x_{\lambda,\mu}(t),u_{\lambda,\mu}(t))\in \mathcal{L}^2_{ad}([0,T]\times \Omega,\mathbb{R}^p)\times U_{\mu}[0,T]$ be the unique solution of system ${\bf MPS}(\lambda,\mu)$
and $(x_{\lambda_m,\mu_n}(t),u_{\lambda_m,\mu_n}(t))\in \mathcal{L}^2_{ad}([0,T]\times \Omega,\mathbb{R}^p)\times U_{\mu_n}[0,T]$ be the unique solution of system ${\bf MPS}(\lambda_m,\mu_n)$. Then $x_{\lambda_m,\mu_n}(t)\rightarrow x_{\lambda,\mu}(t)$ in $\mathcal{L}^2_{ad}([0,T]\times \Omega,\mathbb{R}^p)$ and $u_{\lambda_m,\mu_n}(t)\rightarrow u_{\lambda,\mu}(t)$ in $H[0,T]$.
\end{thm}

\begin{proof}
By Theorem \ref{th1}, the perturbed system ${\bf MPS}(\lambda_m,\mu_n)$ admits unique solution $(x_{\lambda_m,\mu_n}(t),u_{\lambda_m,\mu_n}(t))$ and system ${\bf MPS}(\lambda,\mu)$ has unique solution $(x_{\lambda,\mu}(t),u_{\lambda,\mu}(t))$. In case of no confusion, we will omit $t$ in the later proof and write $x_{\lambda_m,\mu_n}(t)$ as $x_{\lambda_m,\mu_n}$ and $u_{\lambda_m,\mu_n}(t)$ as $u_{\lambda_m,\mu_n}$, respectively.

It is obvious that
 \begin{equation}\label{sta1}
 \|u_{\lambda_m,\mu_n}-u_{\lambda,\mu}\|_{H[0,T]}\leq \|u_{\lambda_m,\mu_n}-u_{\lambda,\mu_n}\|_{H[0,T]}+\|u_{\lambda,\mu_n}-u_{\lambda,\mu}\|_{H[0,T]}
 \end{equation}
and
 \begin{equation}\label{sta1x}
 \|x_{\lambda_m,\mu_n}-x_{\lambda,\mu}\|_{H[0,T]}\leq \|x_{\lambda_m,\mu_n}-x_{\lambda,\mu_n}\|_{H[0,T]}+\|x_{\lambda,\mu_n}-x_{\lambda,\mu}\|_{H[0,T]}.
 \end{equation}

Next, we will show that $u_{\lambda_m,\mu_n}\rightarrow u_{\lambda,\mu_n}\;(\lambda_m\rightarrow\lambda)$ and $u_{\lambda,\mu_n}\rightarrow u_{\lambda,\mu}\;(\mu_n\rightarrow\mu)$. To this end, we have the following six steps.

Step 1: For any fixed $\mu_n\in (M,d)$, we prove that
\begin{align*}
&\mathbb{E}\int_0^T\|u_{\lambda_m,\mu_n}-u_{\lambda,\mu_n}\|^2dt\nonumber\\
\leq& \bar{M}\mathbb{E}\int_0^T\|x_{\lambda_m,\mu_n}-x_{\lambda,\mu_n}\|^2dt+\bar{N}\mathbb{E}\int_0^T\|\tilde{F}_{\lambda}(x_{\lambda_m,\mu_n},u_{\lambda_m,\mu_n})-\tilde{F}_{\lambda_m}(x_{\lambda_m,\mu_n},u_{\lambda_m,\mu_n})\|^2dt,
\end{align*}
where $\bar{M},\;\bar{N}$ are constants. In fact, by Lemma \ref{r31}, we have
\begin{equation}\label{sta2}
u_{\lambda_m,\mu_n}=P_{U_{\mu_n}[0,T]}(u_{\lambda_m,\mu_n}-\rho \tilde{F}_{\lambda_m}(x_{\lambda_m,\mu_n},u_{\lambda_m,\mu_n})),\;\forall \lambda_m,\mu_n\in (M,d),
\end{equation}
 where $0<\rho<\frac{2\bar{C}}{L_F^2}.$ Using \eqref{sta2} and Assumption \ref{ass}, one has
\begin{align}\label{sta3}
&\|u_{\lambda_m,\mu_n}-u_{\lambda,\mu_n}\|_{H[0,T]}\nonumber\\
=&\|P_{U_{\mu_n}[0,T]}(u_{\lambda_m,\mu_n}-\rho \tilde{F}_{\lambda_m}(x_{\lambda_m,\mu_n},u_{\lambda_m,\mu_n}))-P_{U_{\mu_n}[0,T]}(u_{\lambda,\mu_n}-\rho \tilde{F}_{\lambda}(x_{\lambda,\mu_n},u_{\lambda,\mu_n}))\|_{H[0,T]}\nonumber\\
\leq& \|(u_{\lambda_m,\mu_n}-\rho \tilde{F}_{\lambda_m}(x_{\lambda_m,\mu_n},u_{\lambda_m,\mu_n}))-(u_{\lambda,\mu_n}-\rho \tilde{F}_{\lambda}(x_{\lambda,\mu_n},u_{\lambda,\mu_n}))\|_{H[0,T]}\nonumber\\
\leq&\|u_{\lambda_m,\mu_n}-u_{\lambda,\mu_n}+\rho\tilde{F}_{\lambda}(x_{\lambda,\mu_n},u_{\lambda,\mu_n})-\rho\tilde{F}_{\lambda}(x_{\lambda,\mu_n},u_{\lambda_m,\mu_n})+\rho\tilde{F}_{\lambda}(x_{\lambda,\mu_n},u_{\lambda_m,\mu_n})\nonumber\\
&\quad-\rho\tilde{F}_{\lambda}(x_{\lambda_m,\mu_n},u_{\lambda_m,\mu_n})+\rho\tilde{F}_{\lambda}(x_{\lambda_m,\mu_n},u_{\lambda_m,\mu_n})-\rho\tilde{F}_{\lambda_m}(x_{\lambda_m,\mu_n},u_{\lambda_m,\mu_n})\|_{H[0,T]}\nonumber\\
\leq& \|u_{\lambda_m,\mu_n}-u_{\lambda,\mu_n}+\rho\tilde{F}_{\lambda}(x_{\lambda,\mu_n},u_{\lambda,\mu_n})-\rho\tilde{F}_{\lambda}(x_{\lambda,\mu_n},u_{\lambda_m,\mu_n})\|_{H[0,T]}\nonumber\\
&\quad+\rho\|\tilde{F}_{\lambda}(x_{\lambda,\mu_n},u_{\lambda_m,\mu_n})-\tilde{F}_{\lambda}(x_{\lambda_m,\mu_n},u_{\lambda_m,\mu_n})\|_{H[0,T]}\nonumber\\
&\quad+\rho\|\tilde{F}_{\lambda}(x_{\lambda_m,\mu_n},u_{\lambda_m,\mu_n})-\tilde{F}_{\lambda_m}(x_{\lambda_m,\mu_n},u_{\lambda_m,\mu_n})\|_{H[0,T]}\nonumber\\
\leq&\|u_{\lambda_m,\mu_n}-u_{\lambda,\mu_n}+\rho\tilde{F}_{\lambda}(x_{\lambda,\mu_n},u_{\lambda,\mu_n})-\rho\tilde{F}_{\lambda}(x_{\lambda,\mu_n},u_{\lambda_m,\mu_n})\|_{H[0,T]}\nonumber\\
&\quad+\rho L_F\|x_{\lambda_m,\mu_n}-x_{\lambda,\mu_n}\|_{H[0,T]}+\rho\|\tilde{F}_{\lambda}(x_{\lambda_m,\mu_n},u_{\lambda_m,\mu_n})-\tilde{F}_{\lambda_m}(x_{\lambda_m,\mu_n},u_{\lambda_m,\mu_n})\|_{H[0,T]}\nonumber\\
\leq&\sqrt{\|u_{\lambda_m,\mu_n}-u_{\lambda,\mu_n}+\rho\tilde{F}_{\lambda}(x_{\lambda,\mu_n},u_{\lambda,\mu_n})-\rho\tilde{F}_{\lambda}(x_{\lambda,\mu_n},u_{\lambda_m,\mu_n})\|_{H[0,T]}^2}\nonumber\\
&\quad+\rho L_F\|x_{\lambda_m,\mu_n}-x_{\lambda,\mu_n}\|_{H[0,T]}+\rho\|\tilde{F}_{\lambda}(x_{\lambda_m,\mu_n},u_{\lambda_m,\mu_n})
-\tilde{F}_{\lambda_m}(x_{\lambda_m,\mu_n},u_{\lambda_m,\mu_n})\|_{H[0,T]}
\end{align}
and
\begin{align}\label{sta4}
&\|u_{\lambda_m,\mu_n}-u_{\lambda,\mu_n}+\rho\tilde{F}_{\lambda}(x_{\lambda,\mu_n},u_{\lambda,\mu_n})-\rho\tilde{F}_{\lambda}(x_{\lambda,\mu_n},u_{\lambda_m,\mu_n})\|_{H[0,T]}^2\nonumber\\
=& \|u_{\lambda_m,\mu_n}-u_{\lambda,\mu_n}\|^2_{H[0,T]}+\|\rho\tilde{F}_{\lambda}(x_{\lambda,\mu_n},u_{\lambda,\mu_n})-\rho\tilde{F}_{\lambda}(x_{\lambda,\mu_n},u_{\lambda_m,\mu_n})\|_{H[0,T]}^2\nonumber\\
\quad\quad&-2\rho\left\langle\tilde{F}_{\lambda}(x_{\lambda,\mu_n},u_{\lambda_m,\mu_n})-\tilde{F}_{\lambda}(x_{\lambda,\mu_n},u_{\lambda,\mu_n}),u_{\lambda_m,\mu_n}-u_{\lambda,\mu_n}\right\rangle\nonumber\\
\leq& (1+\rho^2L_F^2-2\rho\bar{C})\|u_{\lambda_m,\mu_n}-u_{\lambda,\mu_n}\|^2_{H[0,T]}.
\end{align}
It follows from \eqref{sta3} and \eqref{sta4} that
\begin{align}\label{sta5}
&\|u_{\lambda_m,\mu_n}-u_{\lambda,\mu_n}\|_{H[0,T]}\nonumber\\
\leq& \sqrt{1+\rho^2L_F^2-2\rho\bar{C}}\|u_{\lambda_m,\mu_n}-u_{\lambda,\mu_n}\|_{H[0,T]}+\rho L_F\|x_{\lambda_m,\mu_n}-x_{\lambda,\mu_n}\|_{H[0,T]}\nonumber\\
\quad\quad&+\rho\|\tilde{F}_{\lambda}(x_{\lambda_m,\mu_n},u_{\lambda_m,\mu_n})-\tilde{F}_{\lambda_m}(x_{\lambda_m,\mu_n},u_{\lambda_m,\mu_n})\|_{H[0,T]},
\end{align}
and it leads to
\begin{align}\label{sta6}
&\mathbb{E}\int_0^T\|u_{\lambda_m,\mu_n}-u_{\lambda,\mu_n}\|^2dt\nonumber\\
\leq& \bar{M}\mathbb{E}\int_0^T\|x_{\lambda_m,\mu_n}-x_{\lambda,\mu_n}\|^2dt
+\bar{N}\mathbb{E}\int_0^T\|\tilde{F}_{\lambda}(x_{\lambda_m,\mu_n},u_{\lambda_m,\mu_n})-\tilde{F}_{\lambda_m}(x_{\lambda_m,\mu_n},u_{\lambda_m,\mu_n})\|^2dt,
\end{align}
where
$$
\bar{M}=\frac{2\rho^2L_F^2}{(1-\sqrt{1-2\rho\bar{C}+\rho^2L_F^2})^2}, \quad \bar{N}=\frac{2\rho^2}{(1-\sqrt{1-2\rho\bar{C}+\rho^2L_F^2})^2}.
$$

Step 2: For any fixed $\mu_n\in(M,d)$, one has $x_{\lambda_m,\mu_n}\rightarrow x_{\lambda,\mu_n}\;(\lambda_m\rightarrow\lambda)$. Indeed, according to H\"{o}lder's inequality and the following simple inequality,
\begin{equation}\label{simple}
\left(\left\| \sum_{i=1}^mx_i\right\|\right)^2\leq m\sum_{i=1}^m\| x_i\|^2,
\end{equation}
 we have
 \begin{align}\label{sta7}
 &\mathbb{E}\sup_{t\in[0,\tau]}\|x_{\lambda_m,\mu_n}(t)-x_{\lambda,\mu_n}(t)\|^2\nonumber\\
 \leq& 4\mathbb{E}T\int_0^\tau\| b_{\lambda_m}(s,x_{\lambda_m,\mu_n}(s-),u_{\lambda_m,\mu_n}(s-))-b_{\lambda}(s,x_{\lambda,\mu_n}(s-),u_{\lambda,\mu_n}(s-))\|^2ds\nonumber\\
&\;+4\mathbb{E}\sup_{t\in[0,\tau]}\left(\int_0^t\| \sigma_{\lambda_m}(s,x_{\lambda_m,\mu_n}(s-),u_{\lambda_m,\mu_n}(s-))-\sigma_{\lambda}(s,x_{\lambda,\mu_n}(s-),u_{\lambda,\mu_n}(s-))\| dB(s)\right)^2\nonumber\\
&\;+4\mathbb{E}\sup_{t\in[0,\tau]}\left(\int_0^t\int_{\| y\|<c}\| G_{\lambda_m}(s,x_{\lambda_m,\mu_n}(s-),u_{\lambda_m,\mu_n}(s-),y)-G_{\lambda}(s,x_{\lambda,\mu_n}(s-),u_{\lambda,\mu_n}(s-),y)\|\tilde{N}(ds,dy)\right)^2\nonumber\\
&\; +4\alpha^2\mathbb{E}\sup_{t\in[0,\tau]}\left(\int_0^t(t-s)^{\alpha-1}\|\sigma_{1\lambda_m}(s,x_{\lambda_m,\mu_n}(s-),u_{\lambda_m,\mu_n}(s-))-
\sigma_{1\lambda}(s,x_{\lambda,\mu_n}(s-),u_{\lambda,\mu_n}(s-))\| ds\right)^2\nonumber\\
:=&I_1+I_2+I_3+I_4.
 \end{align}
 It derives from Lemmas \ref{rdoob} and \ref{ito}, Remark \ref{levyito}, H\"{o}lder's inequality, and Assumption \ref{ass} that
 \begin{align}
  \label{staI1}I_1&\leq 8\mathbb{E}T\int_0^\tau\| b_{\lambda_m}(s,x_{\lambda_m,\mu_n}(s-),u_{\lambda_m,\mu_n}(s-))-b_{\lambda}(s,x_{\lambda_m,\mu_n}(s-),u_{\lambda_m,\mu_n}(s-))\|^2ds\nonumber\\
  &\quad\quad +8\mathbb{E}T\int_0^\tau\| b_{\lambda}(s,x_{\lambda_m,\mu_n}(s-),u_{\lambda_m,\mu_n}(s-))-b_{\lambda}(s,x_{\lambda,\mu_n}(s-),u_{\lambda,\mu_n}(s-))\|^2ds\nonumber\\
  &\leq 8TL_{b_{\lambda}}\mathbb{E}\int_0^\tau\| x_{\lambda_m,\mu_n}(s-)-x_{\lambda,\mu_n}(s-)\|^2+\| u_{\lambda_m,\mu_n}(s-)-u_{\lambda,\mu_n}(s-)\|^2ds\nonumber\\
  &\quad\quad+8\mathbb{E}T\int_0^T\| b_{\lambda_m}(s,x_{\lambda_m,\mu_n}(s-),u_{\lambda_m,\mu_n}(s-))-b_{\lambda}(s,x_{\lambda_m,\mu_n}(s-),u_{\lambda_m,\mu_n}(s-))\|^2ds,\\
  \label{staI2}I_2&\leq 32\mathbb{E}\int_0^\tau\| \sigma_{\lambda_m}(s,x_{\lambda_m,\mu_n}(s-),u_{\lambda_m,\mu_n}(s-))-\sigma_{\lambda}(s,x_{\lambda_m,\mu_n}(s-),u_{\lambda_m,\mu_n}(s-))\|^2ds\nonumber\\
  &\quad\quad +32\mathbb{E}\int_0^\tau\| \sigma_{\lambda}(s,x_{\lambda_m,\mu_n}(s-),u_{\lambda_m,\mu_n}(s-))-\sigma_{\lambda}(s,x_{\lambda,\mu_n}(s-),u_{\lambda,\mu_n}(s-))\|^2ds\nonumber\\
  &\leq 32L_{\sigma_{\lambda}}\mathbb{E}\int_0^\tau\| x_{\lambda_m,\mu_n}(s-)-x_{\lambda,\mu_n}(s-)\|^2+\| u_{\lambda_m,\mu_n}(s-)-u_{\lambda,\mu_n}(s-)\|^2ds\nonumber\\
  &\quad\quad +32\mathbb{E}\int_0^T\| \sigma_{\lambda_m}(s,x_{\lambda_m,\mu_n}(s-),u_{\lambda_m,\mu_n}(s-))-\sigma_{\lambda}(s,x_{\lambda_m,\mu_n}(s-),u_{\lambda_m,\mu_n}(s-))\|^2ds,\\
  \label{staI3}I_3&\leq 16\mathbb{E}\int_0^\tau\int_{\| y\|<c}\| G_{\lambda_m}(s,x_{\lambda_m,\mu_n}(s-),u_{\lambda_m,\mu_n}(s-),y)-G_{\lambda}(s,x_{\lambda,\mu_n}(s-),u_{\lambda,\mu_n}(s-),y)\|^2v(dy)ds\nonumber\\
  &\leq 32\mathbb{E}\int_0^\tau\int_{\| y\|<c}\| G_{\lambda_m}(s,x_{\lambda_m,\mu_n}(s-),u_{\lambda_m,\mu_n}(s-),y)-G_{\lambda}(s,x_{\lambda_m,\mu_n}(s-),u_{\lambda_m,\mu_n}(s-),y)\|^2v(dy)ds\nonumber\\
  &\quad\quad+32\mathbb{E}\int_0^\tau\int_{\| y\|<c}\| G_{\lambda}(s,x_{\lambda_m,\mu_n}(s-),u_{\lambda_m,\mu_n}(s-),y)-G_{\lambda}(s,x_{\lambda,\mu_n}(s-),u_{\lambda,\mu_n}(s-),y)\|^2v(dy)ds\nonumber\\
  &\leq 32L_{G_{\lambda}} \mathbb{E}\int_0^\tau\| x_{\lambda_m,\mu_n}(s-)-x_{\lambda,\mu_n}(s-)\|^2+\| u_{\lambda_m,\mu_n}(s-)-u_{\lambda,\mu_n}(s-)\|^2ds\nonumber\\
  &\quad\quad +32\mathbb{E}\int_0^T\int_{\| y\|<c}\| G_{\lambda_m}(s,x_{\lambda_m,\mu_n}(s-),u_{\lambda_m,\mu_n}(s-),y)-G_{\lambda}(s,x_{\lambda_m,\mu_n}(s-),u_{\lambda_m,\mu_n}(s-),y)\|^2v(dy)ds,\\
  \label{staI4}I_4&\leq  4\alpha^2\frac{T^{2\alpha-1}}{2\alpha-1}\mathbb{E}\int_0^\tau\|\sigma_{1\lambda_m}(s,x_{\lambda_m,\mu_n}(s-),u_{\lambda_m,\mu_n}(s-))-
\sigma_{1\lambda}(s,x_{\lambda,\mu_n}(s-),u_{\lambda,\mu_n}(s-))\|^2 ds\nonumber\\
  &\leq 8\alpha^2\frac{T^{2\alpha-1}}{2\alpha-1}L_{\sigma_{1\lambda}}\mathbb{E}\int_0^\tau\| x_{\lambda_m,\mu_n}(s-)-x_{\lambda,\mu_n}(s-)\|^2+\| u_{\lambda_m,\mu_n}(s-)-u_{\lambda,\mu_n}(s-)\|^2ds\nonumber\\
  &\quad\quad+8\alpha^2\frac{T^{2\alpha-1}}{2\alpha-1}\mathbb{E}\int_0^T\|\sigma_{1\lambda_m}(s,x_{\lambda_m,\mu_n}(s-),u_{\lambda_m,\mu_n}(s-))-
\sigma_{1\lambda}(s,x_{\lambda_m,\mu_n}(s-),u_{\lambda_m,\mu_n}(s-))\|^2 ds.
\end{align}
By Step 1, we have
\begin{align}
\label{stai1}I_1&\leq 8TL_{b_{\lambda}}(1+\bar{M})\mathbb{E}\int_0^t\| x_{\lambda_m,\mu_n}(s-)-x_{\lambda,\mu_n}(s-)\|^2ds\nonumber\\
&\quad\quad+8TL_{b_{\lambda}}\bar{N}\mathbb{E}\int_0^T\|\tilde{F}_{\lambda}(x_{\lambda_m,\mu_n},u_{\lambda_m,\mu_n})-\tilde{F}_{\lambda_m}(x_{\lambda_m,\mu_n},u_{\lambda_m,\mu_n})\|^2ds\nonumber\\
&\quad\quad+8\mathbb{E}T\int_0^T\| b_{\lambda_m}(s,x_{\lambda_m,\mu_n}(s-),u_{\lambda_m,\mu_n}(s-))-b_{\lambda}(s,x_{\lambda_m,\mu_n}(s-),u_{\lambda_m,\mu_n}(s-))\|^2ds\nonumber\\
&\leq 8TL_{b_{\lambda}}(1+\bar{M})\int_0^t\mathbb{E}\sup_{\mu\in[0,s]}\| x_{\lambda_m,\mu_n}(\mu)-x_{\lambda,\mu_n}(\mu)\|^2ds\nonumber\\
&\quad\quad+8TL_{b_{\lambda}}\bar{N}\mathbb{E}\int_0^T\|\tilde{F}_{\lambda}(x_{\lambda_m,\mu_n},u_{\lambda_m,\mu_n})-\tilde{F}_{\lambda_m}(x_{\lambda_m,\mu_n},u_{\lambda_m,\mu_n})\|^2ds\nonumber\\
&\quad\quad+8\mathbb{E}T\int_0^T\| b_{\lambda_m}(s,x_{\lambda_m,\mu_n}(s-),u_{\lambda_m,\mu_n}(s-))-b_{\lambda}(s,x_{\lambda_m,\mu_n}(s-),u_{\lambda_m,\mu_n}(s-))\|^2ds,\\
\label{stai2}I_2&\leq 32L_{\sigma_{\lambda}}(1+\bar{M})\mathbb{E}\int_0^t\| x_{\lambda_m,\mu_n}(s-)-x_{\lambda,\mu_n}(s-)\|^2ds\nonumber\\
&\quad\quad+32L_{\sigma_{\lambda}}\bar{N}\mathbb{E}\int_0^T\|\tilde{F}_{\lambda}(x_{\lambda_m,\mu_n},u_{\lambda_m,\mu_n})-\tilde{F}_{\lambda_m}(x_{\lambda_m,\mu_n},u_{\lambda_m,\mu_n})\|^2ds\nonumber\\
&\quad\quad+32\mathbb{E}\int_0^T\| \sigma_{\lambda_m}(s,x_{\lambda_m,\mu_n}(s-),u_{\lambda_m,\mu_n}(s-))-\sigma_{\lambda}(s,x_{\lambda_m,\mu_n}(s-),u_{\lambda_m,\mu_n}(s-))\|^2ds\nonumber\\
&\leq 32L_{\sigma_{\lambda}}(1+\bar{M})\int_0^t\mathbb{E}\sup_{\mu\in[0,s]}\| x_{\lambda_m,\mu_n}(\mu)-x_{\lambda,\mu_n}(\mu)\|^2ds\nonumber\\
&\quad\quad+32L_{\sigma_{\lambda}}\bar{N}\mathbb{E}\int_0^T\|\tilde{F}_{\lambda}(x_{\lambda_m,\mu_n},u_{\lambda_m,\mu_n})-\tilde{F}_{\lambda_m}(x_{\lambda_m,\mu_n},u_{\lambda_m,\mu_n})\|^2ds\nonumber\\
&\quad\quad+32\mathbb{E}\int_0^T\| \sigma_{\lambda_m}(s,x_{\lambda_m,\mu_n}(s-),u_{\lambda_m,\mu_n}(s-))-\sigma_{\lambda}(s,x_{\lambda_m,\mu_n}(s-),u_{\lambda_m,\mu_n}(s-))\|^2ds,\\
\label{stai3}I_3&\leq 32L_{G_{\lambda}}(1+\bar{M})\mathbb{E}\int_0^t\| x_{\lambda_m,\mu_n}(s-)-x_{\lambda,\mu_n}(s-)\|^2ds\nonumber\\
&\quad\quad+32L_{G_{\lambda}}\bar{N}\mathbb{E}\int_0^T\|\tilde{F}_{\lambda}(x_{\lambda_m,\mu_n},u_{\lambda_m,\mu_n})-\tilde{F}_{\lambda_m}(x_{\lambda_m,\mu_n},u_{\lambda_m,\mu_n})\|^2ds\nonumber\\
&\quad\quad+32\mathbb{E}\int_0^T\int_{\| y\|<c}\| G_{\lambda_m}(s,x_{\lambda_m,\mu_n}(s-),u_{\lambda_m,\mu_n}(s-),y)-G_{\lambda}(s,x_{\lambda_m,\mu_n}(s-),u_{\lambda_m,\mu_n}(s-),y)\|^2v(dy)ds\nonumber\\
&\leq 32L_{G_{\lambda}}(1+\bar{M})\int_0^t\mathbb{E}\sup_{\mu\in[0,s]}\| x_{\lambda_m,\mu_n}(\mu)-x_{\lambda,\mu_n}(\mu)\|^2ds\nonumber\\
&\quad\quad+32L_{G_{\lambda}}\bar{N}\mathbb{E}\int_0^T\|\tilde{F}_{\lambda}(x_{\lambda_m,\mu_n},u_{\lambda_m,\mu_n})-\tilde{F}_{\lambda_m}(x_{\lambda_m,\mu_n},u_{\lambda_m,\mu_n})\|^2ds\nonumber\\
&\quad\quad+32\mathbb{E}\int_0^T\int_{\| y\|<c}\| G_{\lambda_m}(s,x_{\lambda_m,\mu_n}(s-),u_{\lambda_m,\mu_n}(s-),y)-G_{\lambda}(s,x_{\lambda_m,\mu_n}(s-),u_{\lambda_m,\mu_n}(s-),y)\|^2v(dy)ds,\\
\label{stai4}I_4&\leq 8\alpha^2\frac{T^{2\alpha-1}}{2\alpha-1}L_{\sigma_{1\lambda}}(1+\bar{M})\mathbb{E}\int_0^t\| x_{\lambda_m,\mu_n}(s-)-x_{\lambda,\mu_n}(s-)\|^2ds\nonumber\\
&\quad\quad+8\alpha^2\frac{T^{2\alpha-1}}{2\alpha-1}L_{\sigma_{1\lambda}}\bar{N}\mathbb{E}\int_0^T\|\tilde{F}_{\lambda}(x_{\lambda_m,\mu_n},u_{\lambda_m,\mu_n})-\tilde{F}_{\lambda_m}(x_{\lambda_m,\mu_n},u_{\lambda_m,\mu_n})\|^2ds\nonumber\\
&\quad\quad+8\alpha^2\frac{T^{2\alpha-1}}{2\alpha-1}\mathbb{E}\int_0^T\|\sigma_{1\lambda_m}(s,x_{\lambda_m,\mu_n}(s-),u_{\lambda_m,\mu_n}(s-))-
\sigma_{1\lambda}(s,x_{\lambda_m,\mu_n}(s-),u_{\lambda_m,\mu_n}(s-))\|^2 ds\nonumber\\
&\leq 8\alpha^2\frac{T^{2\alpha-1}}{2\alpha-1}L_{\sigma_{1\lambda}}(1+\bar{M})\int_0^t\mathbb{E}\sup_{\mu\in[0,s]}\| x_{\lambda_m,\mu_n}(\mu)-x_{\lambda,\mu_n}(\mu)\|^2ds\nonumber\\
&\quad\quad+8\alpha^2\frac{T^{2\alpha-1}}{2\alpha-1}L_{\sigma_{1\lambda}}\bar{N}\mathbb{E}\int_0^T\|\tilde{F}_{\lambda}(x_{\lambda_m,\mu_n},u_{\lambda_m,\mu_n})-\tilde{F}_{\lambda_m}(x_{\lambda_m,\mu_n},u_{\lambda_m,\mu_n})\|^2ds\nonumber\\
&\quad\quad+8\alpha^2\frac{T^{2\alpha-1}}{2\alpha-1}\mathbb{E}\int_0^T\|\sigma_{1\lambda_m}(s,x_{\lambda_m,\mu_n}(s-),u_{\lambda_m,\mu_n}(s-))-
\sigma_{1\lambda}(s,x_{\lambda_m,\mu_n}(s-),u_{\lambda_m,\mu_n}(s-))\|^2 ds.
\end{align}
It follows form \eqref{sta7}, \eqref{stai1}-\eqref{stai4} and Gronwall's inequality that
\begin{equation}\label{sta8}
\mathbb{E}\sup_{t\in[0,\tau]}\|x_{\lambda_m,\mu_n}(t)-x_{\lambda,\mu_n}(t)\|^2
\leq A(\lambda_m)\exp\left(B(0)T\right),
\end{equation}
where
\begin{align*}
&A(\lambda_m)\nonumber\\
=&\left(8TL_{b_{\lambda}}+32L_{\sigma_{\lambda}}+32L_{G_{\lambda}}+8\alpha^2\frac{T^{2\alpha-1}}{2\alpha-1}L_{\sigma_{1\lambda}}\right)\bar{N}\mathbb{E}\int_0^T\|\tilde{F}_{\lambda}(x_{\lambda_m,\mu_n},u_{\lambda_m,\mu_n})-\tilde{F}_{\lambda_m}(x_{\lambda_m,\mu_n},u_{\lambda_m,\mu_n})\|^2ds\nonumber\\
&\quad+8\mathbb{E}T\int_0^T\| b_{\lambda_m}(s,x_{\lambda_m,\mu_n}(s-),u_{\lambda_m,\mu_n}(s-))-b_{\lambda}(s,x_{\lambda_m,\mu_n}(s-),u_{\lambda_m,\mu_n}(s-))\|^2ds\nonumber\\
&\quad+32\mathbb{E}\int_0^T\| \sigma_{\lambda_m}(s,x_{\lambda_m,\mu_n}(s-),u_{\lambda_m,\mu_n}(s-))-\sigma_{\lambda}(s,x_{\lambda_m,\mu_n}(s-),u_{\lambda_m,\mu_n}(s-))\|^2ds\nonumber\\
&\quad+32\mathbb{E}\int_0^T\int_{\| y\|<c}\| G_{\lambda_m}(s,x_{\lambda_m,\mu_n}(s-),u_{\lambda_m,\mu_n}(s-),y)-G_{\lambda}(s,x_{\lambda_m,\mu_n}(s-),u_{\lambda_m,\mu_n}(s-),y)\|^2v(dy)ds\nonumber\\
&\quad+8\alpha^2\frac{T^{2\alpha-1}}{2\alpha-1}\mathbb{E}\int_0^T\|\sigma_{1\lambda_m}(s,x_{\lambda_m,\mu_n}(s-),u_{\lambda_m,\mu_n}(s-))-
\sigma_{1\lambda}(s,x_{\lambda_m,\mu_n}(s-),u_{\lambda_m,\mu_n}(s-))\|^2 ds
\end{align*}
and
$$
B(0)=\left(8TL_{b_{\lambda}}+32L_{\sigma_{\lambda}}+32L_{G_{\lambda}}
+8\alpha^2\frac{T^{2\alpha-1}}{2\alpha-1}L_{\sigma_{1\lambda}}\right)(1+\bar{M}).
$$
Thus, the condition (iv) of Assumption \ref{ass} shows that $\lim_{\lambda_m\rightarrow\lambda}A(\lambda_m)=0$ and so it can be inferred that $x_{\lambda_m,\mu_n}\rightarrow x_{\lambda,\mu_n}\;(\lambda_m\rightarrow\lambda)$.

Step 3: We claim that, for any fixed $\mu_n\in(M,d)$, $u_{\lambda_m,\mu_n}\rightarrow u_{\lambda,\mu_n}\;(\lambda_m\rightarrow\lambda)$. In fact, by Step 1, Step 2 and the condition (iv) of Assumption \ref{ass}, the claim is clearly true.

Step 4: We show that
\begin{align}\label{sta12}
&\mathbb{E}\int_0^T\|u_{\lambda,\mu_n}-u_{\lambda,\mu}\|^2dt\leq  \hat{M}\mathbb{E}\int_0^T\|x_{\lambda,\mu_n}-x_{\lambda,\mu}\|^2dt\nonumber\\
&\qquad \mbox{} +\hat{N}\mathbb{E}\int_0^T\|P_{U_{\mu_n}[0,T]}(u_{\lambda,\mu}-\rho \tilde{F}_{\lambda}(x_{\lambda,\mu},u_{\lambda,\mu}))-P_{U_{\mu}[0,T]}(u_{\lambda,\mu}-\rho \tilde{F}_{\lambda}(x_{\lambda,\mu},u_{\lambda,\mu}))\|^2dt,
\end{align}
where $$
\hat{M}=\frac{2\rho^2L_F^2}{(1-\sqrt{1-2\rho\bar{C}+\rho^2L_F^2})^2}, \hat{N}=\frac{2}{(1-\sqrt{1-2\rho\bar{C}+\rho^2L_F^2})^2}.
$$
Indeed, similar to \eqref{sta3} and \eqref{sta4}, one has
\begin{align}\label{sta11}
& \|u_{\lambda,\mu_n}-u_{\lambda,\mu}\|_{H[0,T]}\nonumber\\
=&\|P_{U_{\mu_n}[0,T]}(u_{\lambda,\mu_n}-\rho \tilde{F}_{\lambda}(x_{\lambda,\mu_n},u_{\lambda,\mu_n}))-P_{U_{\mu}[0,T]}(u_{\lambda,\mu}-\rho \tilde{F}_{\lambda}(x_{\lambda,\mu},u_{\lambda,\mu}))\|_{H[0,T]}\nonumber\\
\leq & \|P_{U_{\mu_n}[0,T]}(u_{\lambda,\mu_n}-\rho \tilde{F}_{\lambda}(x_{\lambda,\mu_n},u_{\lambda,\mu_n}))-P_{U_{\mu_n}[0,T]}(u_{\lambda,\mu}-\rho \tilde{F}_{\lambda}(x_{\lambda,\mu},u_{\lambda,\mu}))\|_{H[0,T]}\nonumber\\
&\quad\quad \mbox{}+\|P_{U_{\mu_n}[0,T]}(u_{\lambda,\mu}-\rho \tilde{F}_{\lambda}(x_{\lambda,\mu},u_{\lambda,\mu}))-P_{U_{\mu}[0,T]}(u_{\lambda,\mu}-\rho \tilde{F}_{\lambda}(x_{\lambda,\mu},u_{\lambda,\mu}))\|_{H[0,T]}\nonumber\\
\leq & \|u_{\lambda,\mu_n}-u_{\lambda,\mu}-\rho \tilde{F}_{\lambda}(x_{\lambda,\mu_n},u_{\lambda,\mu_n})+\rho \tilde{F}_{\lambda}(x_{\lambda,\mu},u_{\lambda,\mu})\|_{H[0,T]}\nonumber\\
&\quad\quad\mbox{} +\|P_{U_{\mu_n}[0,T]}(u_{\lambda,\mu}-\rho \tilde{F}_{\lambda}(x_{\lambda,\mu},u_{\lambda,\mu}))-P_{U_{\mu}[0,T]}(u_{\lambda,\mu}-\rho \tilde{F}_{\lambda}(x_{\lambda,\mu},u_{\lambda,\mu}))\|_{H[0,T]}\nonumber\\
\leq& \|u_{\lambda,\mu_n}-u_{\lambda,\mu}-\rho \tilde{F}_{\lambda}(x_{\lambda,\mu_n},u_{\lambda,\mu_n})+\rho \tilde{F}_{\lambda}(x_{\lambda,\mu_n},u_{\lambda,\mu})\|_{H[0,T]}\nonumber\\
&\quad\quad\mbox{} +\|-\rho \tilde{F}_{\lambda}(x_{\lambda,\mu_n},u_{\lambda,\mu_n})+\rho \tilde{F}_{\lambda}(x_{\lambda,\mu},u_{\lambda,\mu})\|_{H[0,T]}\nonumber\\
&\quad\quad\mbox{} +\|P_{U_{\mu_n}[0,T]}(u_{\lambda,\mu}-\rho \tilde{F}_{\lambda}(x_{\lambda,\mu},u_{\lambda,\mu}))-P_{U_{\mu}[0,T]}(u_{\lambda,\mu}-\rho \tilde{F}_{\lambda}(x_{\lambda,\mu},u_{\lambda,\mu}))\|_{H[0,T]}\nonumber\\
\leq& \sqrt{(1+\rho^2L_F^2-2\rho\bar{C})}\|u_{\lambda,\mu_n}-u_{\lambda,\mu}\|_{H[0,T]}+\rho L_F\|x_{\lambda,\mu_n}-x_{\lambda,\mu}\|_{H[0,T]}\nonumber\\
&\quad\quad\mbox{} +\|P_{U_{\mu_n}[0,T]}(u_{\lambda,\mu}-\rho \tilde{F}_{\lambda}(x_{\lambda,\mu},u_{\lambda,\mu}))-P_{U_{\mu}[0,T]}(u_{\lambda,\mu}-\rho \tilde{F}_{\lambda}(x_{\lambda,\mu},u_{\lambda,\mu}))\|_{H[0,T]}.
\end{align}
It follows that \eqref{sta12} holds.

Step 5: We claim that $x_{\lambda,\mu_n}\rightarrow x_{\lambda,\mu}\;(\mu_n\rightarrow\mu)$. In fact, it derives from \eqref{simple} and H\"{o}lder's inequality that
\begin{align}\label{sta13}
&\mathbb{E}\sup_{t\in[0,\tau]}\|x_{\lambda,\mu_n}(t)-x_{\lambda,\mu}(t)\|^2\nonumber\\
\leq& 4\mathbb{E}T\int_0^\tau\| b_{\lambda}(s,x_{\lambda,\mu_n}(s-),u_{\lambda,\mu_n}(s-))-b_{\lambda}(s,x_{\lambda,\mu}(s-),u_{\lambda,\mu}(s-))\|^2ds\nonumber\\
&\quad +4\mathbb{E}\sup_{t\in[0,\tau]}\left(\int_0^t\| \sigma_{\lambda}(s,x_{\lambda,\mu_n}(s-),u_{\lambda,\mu_n}(s-))-\sigma_{\lambda}(s,x_{\lambda,\mu}(s-),u_{\lambda,\mu}(s-))\| dB(s)\right)^2\nonumber\\
&\quad +4\mathbb{E}\sup_{t\in[0,\tau]}\left(\int_0^t\int_{\| y\|<c}\| G_{\lambda}(s,x_{\lambda,\mu_n}(s-),u_{\lambda,\mu_n}(s-),y)-G_{\lambda}(s,x_{\lambda,\mu}(s-),u_{\lambda,\mu}(s-),y)\|\tilde{N}(ds,dy)\right)^2\nonumber\\
&\quad +4\alpha^2\mathbb{E}\sup_{t\in[0,\tau]}\left(\int_0^t(t-s)^{\alpha-1}\|\sigma_{1\lambda}(s,x_{\lambda,\mu_n}(s-),u_{\lambda,\mu_n}(s-))-
\sigma_{1\lambda}(s,x_{\lambda,\mu}(s-),u_{\lambda,\mu}(s-))\| ds\right)^2\nonumber\\
&:=J_1+J_2+J_3+J_4.
\end{align}
According to Lemma \ref{rdoob}, Lemma \ref{ito}, Remark \ref{levyito}, H\"{o}lder inequality, Step 4 and Assumption \ref{ass}, one has
\begin{align}
  \label{staJ1}J_1&\leq 4TL_{b_{\lambda}}(1+\hat{M})\mathbb{E}\int_0^t\| x_{\lambda,\mu_n}(s-)-x_{\lambda,\mu}(s-)\|^2ds\nonumber\\
  &\quad\quad+4TL_{b_{\lambda}}\hat{N}\mathbb{E}\int_0^T\|P_{U_{\mu_n}[0,T]}(u_{\lambda,\mu}-\rho \tilde{F}_{\lambda}(x_{\lambda,\mu},u_{\lambda,\mu}))-P_{U_{\mu}[0,T]}(u_{\lambda,\mu}-\rho \tilde{F}_{\lambda}(x_{\lambda,\mu},u_{\lambda,\mu}))\|^2ds,\\
  \label{staJ2}J_2&\leq 16\mathbb{E}\int_0^\tau\| \sigma_{\lambda}(s,x_{\lambda,\mu_n}(s-),u_{\lambda,\mu_n}(s-))-\sigma_{\lambda}(s,x_{\lambda,\mu}(s-),u_{\lambda,\mu}(s-))\|^2ds\nonumber\\
  &\leq 16L_{\sigma_{\lambda}}(1+\hat{M})\mathbb{E}\int_0^\tau\| x_{\lambda,\mu_n}(s-)-x_{\lambda,\mu}(s-)\|^2ds\nonumber\\
  &\quad\quad+16L_{\sigma_{\lambda}}\hat{N}\mathbb{E}\int_0^T\|P_{U_{\mu_n}[0,T]}(u_{\lambda,\mu}-\rho \tilde{F}_{\lambda}(x_{\lambda,\mu},u_{\lambda,\mu}))-P_{U_{\mu}[0,T]}(u_{\lambda,\mu}-\rho \tilde{F}_{\lambda}(x_{\lambda,\mu},u_{\lambda,\mu}))\|^2ds,\\
  \label{staJ3}J_3&\leq 16\mathbb{E}\int_0^\tau\int_{\| y\|<c}\| G_{\lambda}(s,x_{\lambda,\mu_n}(s-),u_{\lambda,\mu_n}(s-),y)-G_{\lambda}(s,x_{\lambda,\mu}(s-),u_{\lambda,\mu}(s-),y)\|^2v(dy)ds\nonumber\\
  &\leq 16L_{G_{\lambda}}(1+\hat{M}) \mathbb{E}\int_0^\tau\| x_{\lambda,\mu_n}(s-)-x_{\lambda,\mu}(s-)\|^2ds\nonumber\\
  &\quad\quad +16L_{G_{\lambda}}\hat{N}\mathbb{E}\int_0^T\|P_{U_{\mu_n}[0,T]}(u_{\lambda,\mu}-\rho \tilde{F}_{\lambda}(x_{\lambda,\mu},u_{\lambda,\mu}))-P_{U_{\mu}[0,T]}(u_{\lambda,\mu}-\rho \tilde{F}_{\lambda}(x_{\lambda,\mu},u_{\lambda,\mu}))\|^2ds,\\
  \label{staJ4}J_4&\leq  4\alpha^2\frac{T^{2\alpha-1}}{2\alpha-1}\mathbb{E}\int_0^\tau\|\sigma_{1\lambda}(s,x_{\lambda,\mu_n}(s-),u_{\lambda,\mu_n}(s-))-
\sigma_{1\lambda}(s,x_{\lambda,\mu}(s-),u_{\lambda,\mu}(s-))\|^2 ds\nonumber\\
  &\leq 4\alpha^2\frac{T^{2\alpha-1}}{2\alpha-1}L_{\sigma_{1\lambda}}(1+\hat{M})\mathbb{E}\int_0^\tau\| x_{\lambda,\mu_n}(s-)-x_{\lambda,\mu}(s-)\|^2ds\nonumber\\
  &\quad\quad+4\alpha^2\frac{T^{2\alpha-1}}{2\alpha-1}L_{\sigma_{1\lambda}}\hat{N}\mathbb{E}\int_0^T\|P_{U_{\mu_n}[0,T]}(u_{\lambda,\mu}-\rho \tilde{F}_{\lambda}(x_{\lambda,\mu},u_{\lambda,\mu}))-P_{U_{\mu}[0,T]}(u_{\lambda,\mu}-\rho \tilde{F}_{\lambda}(x_{\lambda,\mu},u_{\lambda,\mu}))\|^2ds.
\end{align}
It follows from \eqref{sta13}-\eqref{staJ4} and Gronwall's inequality that
\begin{equation*}\label{sta14}
\mathbb{E}\sup_{t\in[0,\tau]}\|x_{\lambda,\mu_n}(t)-x_{\lambda,\mu}(t)\|^2
\leq Z(n)\exp\left(D(0)T\right),
\end{equation*}
where
$$
\begin{cases}
Z(\mu_n)=\bar{Z}\hat{N}\mathbb{E}\int_0^T\|P_{U_{\mu_n}[0,T]}(u_{\lambda,\mu}-\rho \tilde{F}_{\lambda}(x_{\lambda,\mu},u_{\lambda,\mu}))-P_{U_{\mu}[0,T]}(u_{\lambda,\mu}-\rho \tilde{F}_{\lambda}(x_{\lambda,\mu},u_{\lambda,\mu}))\|^2ds,\\
B(0)=\bar{Z}(1+\bar{M}),\\
\bar{Z}=4TL_{b_{\lambda}}+16L_{\sigma_{\lambda}}+16L_{G_{\lambda}}+4\alpha^2\frac{T^{2\alpha-1}}{2\alpha-1}L_{\sigma_{1\lambda}}.
\end{cases}
$$
Moreover, Theorem \ref{pro} shows that $\lim_{\mu_n\rightarrow\mu}Z(\mu_n)=0$. Therefore, it can be inferred that $x_{\lambda,\mu_n}\rightarrow x_{\lambda,\mu}\;(\mu_n\rightarrow\mu)$.

Step 6: We claim that $u_{\lambda,\mu_n}\rightarrow u_{\lambda,\mu}\;(\mu_n\rightarrow\mu)$. Indeed, it can be easily verified by Steps 4 and 5, and Theorem \ref{pro}.

Finally, according to Steps 1-6 and \eqref{sta1}-\eqref{sta1x}, we have $x_{\lambda_m,\mu_n}\rightarrow x_{\lambda,\mu}\;(\lambda_m\rightarrow\lambda,\mu_n\rightarrow\mu)$ in $\mathcal{L}^2_{ad}([0,T]\times \Omega,\mathbb{R}^p)$ and $u_{\lambda_m,\mu_n}\rightarrow u_{\lambda,\mu}\;(\lambda_m\rightarrow\lambda,\mu_n\rightarrow\mu)$ in $H[0,T]$. That complete our proof.
\end{proof}

\begin{remark}\label{zhang}
Clearly, if $\sigma_{1\lambda}\equiv 0$ and $G_{\lambda}\equiv 0$, then the system \eqref{S5} reduces to the following multi-parameter system associated with the system \eqref{S1}:
\begin{align}\label{AS1}
\begin{cases}
 dx_{\lambda,\mu}(t)=b_{\lambda}(t,x_{\lambda,\mu}(t),u_{\lambda,\mu}(t))dt+\sigma_{\lambda}(t,x(t),u(t))dB(t), \quad
 x_{\lambda,\mu}(0)=p_0,\\
 \langle F_{\lambda}(t,\omega,x_{\lambda,\mu}(t,\omega),u_{\lambda,\mu}(t,\omega)),v-u_{\lambda,\mu}(t,\omega)\rangle\geq0, \; \forall v\in K_{\mu}, \; a.e. \, t\in [0,T], \; a.s. \, \omega\in\Omega.
 \end{cases}
\end{align}
and our stability result is still new to SDVI proposed in \cite{zhang2023}. Moreover, if $\sigma_{1\lambda}\equiv 0$, $G_{\lambda}\equiv 0$ and the constraint set $K_{\mu}\equiv K$, where $K$ is a fixed closed convex subset of $\mathbb{R}^q$, then Theorem \ref{multi} is consistent with Theorem 4.1 in \cite{zhang2023}.
\end{remark}

From equivalence of the stochastic fractional differential variational inequality with L\'{e}vy jump and stochastic fractional differential complementarity problem with L\'{e}vy jump, we can obtain the following unique solvability  for the stochastic fractional differential complementarity problem with L\'{e}vy jump by using Theorem \ref{th1}.
\begin{cor}
If conditions (i)-(iii) of Assumption \ref{ass} are satisfied and $K$ is a closed convex cone in $\mathbb{R}^q$ and $K^*$ is the dual cone of $K$ then for any $\lambda,\mu\in(M,d)$, the following system ${\bf SFDCP}(\lambda,\mu)$
\begin{align*}
  \begin{cases}
 dx_{\lambda,\mu}(t)=b_{\lambda}(t,x_{\lambda,\mu}(t-),u_{\lambda,\mu}(t-))dt+\sigma_{1\lambda}(t,x_{\lambda,\mu}(t-),u_{\lambda,\mu}(t-))(dt)^\alpha\\
 \qquad\qquad\;\mbox{}+\sigma_{\lambda}(t,x_{\lambda,\mu}(t-),u_{\lambda,\mu}(t-))dB(t)\\
 \qquad\qquad\;\mbox{} +\int_{\| y\|<c}G_{\lambda}(t,x_{\lambda,\mu}(t-),u_{\lambda,\mu}(t-),y)\tilde{N}(dt,dy), \; \alpha \in (\frac{1}{2},1),\\
 x_{\lambda,\mu}(0)=p_0,\\
K_{\mu}\ni u_{\lambda,\mu}(t,\omega)\;\bot\;F_{\lambda}(t,\omega,x_{\lambda,\mu}(t,\omega),u_{\lambda,\mu}(t,\omega))\in K_{\mu}^*, \; a.e. \, t\in [0,T], \; a.s. \, \omega\in\Omega,
 \end{cases}
  \end{align*}
admits a unique solution $(x_{\lambda,\mu}(t),u_{\lambda,\mu}(t))\in \mathcal{L}^2_{ad}([0,T]\times \Omega,\mathbb{R}^p)\times U_{\mu}[0,T]$.
\end{cor}

We also can obtain the stability of solutions for the stochastic fractional differential complementarity problems with L\'{e}vy jump by employing Theorem \ref{multi}.
\begin{cor}
Assume $K$ is a closed convex cone in $\mathbb{R}^q$ and $K^*$ is the dual cone of $K$ and $\lambda_m, \mu_n\in(M,d)$ such that $\lambda_m\rightarrow\lambda$ and $\mu_n\rightarrow\mu$.
  Suppose $b_{\lambda_m},b_{\lambda};\sigma_{\lambda_m},\sigma_{\lambda};\sigma_{1\lambda_m},\sigma_{1\lambda};$ $G_{\lambda_m},G_{\lambda};F_{\lambda_m},F_{\lambda};\tilde{F}_{\lambda_m},\tilde{F}_{\lambda_m}$ satisfy the similar conditions in Assumption \ref{ass}, and $K_{\mu_n}\overset{M}{\rightarrow}K_{\mu}$. Moreover, let $(x_{\lambda,\mu}(t),u_{\lambda,\mu}(t))\in \mathcal{L}^2_{ad}([0,T]\times \Omega,\mathbb{R}^p)\times U_{\mu}[0,T]$ and $(x_{\lambda_m,\mu_n}(t),u_{\lambda_m,\mu_n}(t))\in \mathcal{L}^2_{ad}([0,T]\times \Omega,\mathbb{R}^p)\times U_{\mu_n}[0,T]$  be the unique solution of ${\bf SFDCP}(\lambda,\mu) $ and ${\bf SFDCP}(\lambda_m,\mu_n) $,  respectively. Then $x_{\lambda_m,\mu_n}(t)\rightarrow x_{\lambda,\mu}(t)$ in $\mathcal{L}^2_{ad}([0,T]\times \Omega,\mathbb{R}^p)$ and $u_{\lambda_m,\mu_n}(t)\rightarrow u_{\lambda,\mu}(t)$ in $H[0,T]$.
\end{cor}

\section{Applications}
\setcounter{equation}{0}

In this section, by using the stability result presented in the previous section, we give some stability results for the spatial price equilibrium problem and the multi-agent optimization problem in stochastic environments.

\subsection{The stochastic spatial price equilibrium problem}
The spatial price equilibrium models have been widely studied because of its practical application value in agriculture, energy market, economics and finance \cite{Daniele2004,Li2015,zhang2023}. In this subsection, we specialize the Multi-parameter stability result to the stochastic spatial price equilibrium problem \cite{zeng2024}. To this end, we recall the stochastic spatial price equilibrium problem of a single commodity with memory and jumps in the time period of $[0,T]$ as follows.

\begin{itemize}
  \item $S_i$: the $i$th supply market, $i=1,2,\cdot\cdot\cdot, m$.
  \item $D_j$: the $j$th demand market, $j=1,2,\cdot\cdot\cdot, n$.
  \item $a_{ij}(t,\omega)$: the quantity of commodities conveyed from the supply market $S_i$ to the demand market $D_j$ at time $t$, and $a(t,\omega)=(a_{ij}(t,\omega))\in \mathbb{R}^{m\times n}$.
  \item $\bar{S}_i(t,\omega)=\sum_{j=1}^{n}a_{ij}(t,\omega)$: the amount of commodities provided by  supply market $S_i$ at time $t$, and $\bar{S}(t,\omega)=(\bar{S}_1(t,\omega),\cdots,\bar{S}_m(t,\omega))\in \mathbb{R}^{m}$.
  \item $\bar{D}_j(t,\omega)=\sum_{i=1}^{m}a_{ij}(t,\omega)$: the demand for commodities in demand market $D_j$ at time $t$, and $\bar{D}(t,\omega)=(\bar{D}_1(t,\omega),\cdots,\bar{D}_n(t,\omega))\in \mathbb{R}^{n}$.
  \item $p_i(t,\omega)$: the supply price of commodity associated with supply market $S_i$
      at time $t$, and $p(t,\omega)=(p_1(t,\omega),\cdots,p_m(t,\omega))\in \mathcal{L}^2_{ad}([0,T]\times \Omega,\mathbb{R}^m)$.
  \item $q_j(t,\omega)$: the demand price of commodity associated with demand market $D_j$
      at time $t$, and $q(t,\omega)=(q_1(t,\omega),\cdots,q_n(t,\omega))\in \mathcal{L}^2_{ad}([0,T]\times \Omega,\mathbb{R}^n)$.
  \item $c_{ij}(t,\omega)=c_{ij}(a_{ij}(t,\omega))$: the unit transportation cost from $S_i$ to $D_j$ at time $t$, and $c(t,\omega)=(c_{ij}(t,\omega))\in \mathbb{R}^{m\times n}$.
  \item $\mathcal{L}^2_{ad}= \mathcal{L}^2_{ad}([0,T]\times \Omega, \mathbb{R}^m)\times \mathcal{L}^2_{ad}([0,T]\times \Omega, \mathbb{R}^n)\times \mathcal{L}^2_{ad}([0,T]\times \Omega, \mathbb{R}^{m\times n})$ and
      $$\langle a, b\rangle_{\mathcal{L}^2_{ad}}=\mathbb{E}\int^T_0\langle a(t,\omega), b(t,\omega)\rangle dt, \; \forall a,b\in\mathcal{L}^2_{ad}.$$
  \item $u(t,\omega)=(\bar{S}(t,\omega),\bar{D}(t,\omega),a(t,\omega))\in \mathbb{R}^{m}\times \mathbb{R}^{n}\times \mathbb{R}^{m\times n}$.
  \item $K=\{(A,B,C):A=(A_1,A_2,\cdots,A_m)\in \mathbb{R}^m,\, B=(B_1,B_2,\cdots,B_n)\in\mathbb{R}^n,\, C=(C_{ij})\in\mathbb{R}^{m\times n},\, C_{ij}\ge 0, \, A_i=\sum_{j=1}^nC_{ij}, B_j=\sum_{i=1}^mC_{ij},\;i=1,2,\cdots, m;\,j=1,2,\cdots,n\}$.
  \item $U[0,T]=\{u\in \mathcal{L}_{ad}^2: u(t,\omega)\in K ,\;a.e.\,t\in[0,T],\; a.s.\,\omega\in\Omega\}$.
  \item For any $(p,q)\in \mathcal{L}^2_{ad}([0,T]\times \Omega,\mathbb{R}^m)\times \mathcal{L}^2_{ad}([0,T]\times \Omega,\mathbb{R}^n)$ and any $ u=(\bar{S},\bar{D},a)\in U_{\mu}[0,T]$, let
$$\tilde{F}(p,q,u)(s,\omega)=F(s,\omega,p(s,\omega),q(s,\omega),u(s,\omega)),\; \forall s\in[0,T],\forall \omega\in\Omega$$
and
\begin{align*}
&\quad \langle \tilde{F}(p,q,u), u\rangle_{\mathcal{L}_{ad}^2}\\
&=\mathbb{E}\int_0^T\langle p(t,\omega), \bar{S}(t,\omega)\rangle-\langle q(t,\omega),\bar{D}(t,\omega)\rangle+\langle c(a(t,\omega)),a(t,\omega)\rangle dt.
\end{align*}
\end{itemize}

The asset price processes $p(t,\omega)$ and $q(t,\omega)$ satisfy the following stochastic fractional differential equations with jumps:
\begin{align}\label{4.1}
  \begin{cases}
  dp(t)=b_1(t,p(t-),\bar{S}(t-))dt+\sigma_1(t,p(t-),\bar{S}(t-))(dt)^\alpha+f_1(t,p(t-),\bar{S}(t-))dB_1(t)\\
  \qquad\quad\;+\int_{\| x\|<c}G_1(t,p(t-),\bar{S}(t-),x)\tilde{N}_1(dt,dx),\\
 p(0)=p_0,\\
 dq(t)=b_2(t,q(t-),\bar{D}(t-))dt+\sigma_2(t,q(t-),\bar{D}(t-))(dt)^\alpha+f_2(t,q(t-),\bar{D}(t-))dB_2(t)\\
 \qquad\quad\;+\int_{\| x\|<c}G_2(t,q(t-),\bar{D}(t-),x)\tilde{N}_2(dt,dx),\\
 q(0)=q_0,
  \end{cases}
\end{align}
where $b_i,\sigma_i,f_i,G_i(i=1,2)$ are suitable measurable functions, $\sigma_1(t,p(t,\omega),\bar{S}(t,\omega))$ and $\sigma_2(t,q(t,\omega),\bar{D}(t,\omega))$ are continuous with respect to $t$, $B_1(t)$ and $B_2(t)$ are two $\mathcal{F}_t$-adapted Brownian motions,  $N_1,N_2$ are both $\mathcal{F}_t$-adapted Poisson measure, and their associated compensated martingale measures are defined by $\widetilde{N}_i(dt,dx) :=N_i(dt,dx)-v_i(dx)dt$ for $i=1,2$.
Moreover, we assume that $N_1,N_2,B_1,B_2$ are independent of each other. From \cite{zhang2023, zeng2024}, we have  the following definition of spatial price equilibrium within a stochastic environment influenced by memory and L\'{e}vy jumps.

\begin{defn}\label{spatial} Let
$u^{\ast}(t,\omega)=(\bar{S}^{\ast}(t,\omega),\bar{D}^{\ast}(t,\omega),a^{\ast}(t,\omega))$ such that $u^{\ast}\in U[0,T]$. Then $u^{\ast}$ is said to be a spatial price equilibrium in a stochastic environment if and only if the following conditions are satisfied:
\begin{align*}
 p^{\ast}_i(t,\omega)+c_{ij}(a_{ij}^{\ast}(t,\omega))
 \begin{cases}
 =q^{\ast}_j(t,\omega), \quad\text{if}\quad a^{\ast}_{ij}\geq0\\
 \geq q^{\ast}_j(t,\omega), \quad\text{if}\quad a^{\ast}_{ij}=0
 \end{cases}
 \;a.e.\,t\in[0,T],\; a.s.\,\omega\in\Omega,
\end{align*}
for $i=1,2,\cdots, m$; $j=1,2,\cdots, n$, where $p^{\ast}(t,\omega)$ and $q^{\ast}(t,\omega)$ satisfy \eqref{4.1}.
\end{defn}

From \cite{zeng2024}, we know that the spatial price equilibrium problem in stochastic environment is equivalent to the following stochastic system ${\bf SPEP}(b,\sigma,f,G,\tilde{F};K,y_0)$:
\begin{equation}\label{SP1}
 \begin{cases}
 dy(t)=b(t,y(t-),u(t-))dt+\sigma(t,y(t-),u(t-))(dt)^\alpha+f(t,y(t-),u(t-))dB(t)\\
 \qquad\qquad+\int_{\| x\|<c}G(t,y(t-),u(t-),x)\tilde{N}(dt,dx),\\
 y(0)=y_0,\\
 \langle \bar{F}(t,\omega,y(t,\omega),u(t,\omega)),v-u(t,\omega)\rangle\geq0,\;\forall v\in K,\;a.e.\,t\in [0,T],\;a.s.\,\omega\in\Omega,\;\alpha \in (\frac{1}{2},1),
 \end{cases}
\end{equation}
where
$$
  \begin{cases}
  y(t)=(p^{\ast}(t),q^{\ast}(t))^T,\; y(0)=(p^{\ast}_0,q^{\ast}_0)^T,\; u(t,\omega)=u^{\ast}(t,\omega),\\
  b(t,y(t),u(t))=(b_1(t,p^{\ast}(t),\bar{S}^{\ast}(t)),b_2(t,q^{\ast}(t),\bar{D}^{\ast}(t)))^T,\\
  \sigma(t,y(t),u(t))=(\sigma_1(t,p^{\ast}(t),\bar{S}^{\ast}(t)),\sigma_2(t,q^{\ast}(t),\bar{D}^{\ast}(t)))^T,\\
  f(t,y(t),u(t))=\left(
                   \begin{array}{cc}
                     f_1(t,p^{\ast}(t),\bar{S}^{\ast}(t)) & 0 \\
                     0 & f_2(t,q^{\ast}(t),\bar{D}^{\ast}(t)) \\
                   \end{array}
                 \right)\\
  G(t,y(t),u(t),x)=\left(
                     \begin{array}{cc}
                       G_1(t,p^{\ast}(t),\bar{S}^{\ast}(t),x) & 0 \\
                       0 & G_2(t,q^{\ast}(t),\bar{D}^{\ast}(t),x) \\
                     \end{array}
                   \right)\\
 B(t)=(B_1(t),B_2(t))^T,\; \tilde{N}(t,x)=(\tilde{N}_1(t,x),\tilde{N}_2(t,x))^T,\\
 \bar{F}(t,\omega,y(t,\omega),u(t,\omega))=F(t,\omega,p^{\ast}(t,\omega),q^{\ast}(t,\omega),u^{\ast}(t,\omega)),
  \end{cases}
$$
More importantly, it is evident that $K$ is a closed convex cone in this subsection, hence \eqref{SP1} also has an equivalent complementary problem as below:
$$
 \begin{cases}
 dy(t)=b(t,y(t-),u(t-))dt+\sigma(t,y(t-),u(t-))(dt)^\alpha+f(t,y(t-),u(t-))dB(t)\\
 \qquad\qquad+\int_{\| x\|<c}G(t,y(t-),u(t-),x)\tilde{N}(dt,dx),\\
 y(0)=y_0,\\
K\ni u(t,\omega)\; \bot\;  \bar{F}(t,\omega,y(t,\omega),u(t,\omega))\in K^*, \; a.e. \, t\in [0,T], \; a.s. \, \omega\in\Omega,\;\alpha \in (\frac{1}{2},1).
 \end{cases}
$$

It should be noted that, in the real world, accurate data are almost impossible to obtain, which leads to a system that is always perturbed. When we have enough data sampled and accurately enough, such perturbations have decreasing effects on the original system, we get more and more sufficient information, and the description of the model becomes more and more accurate, and the mappings of the perturbed system converge to the mappings of the original system, and the constraint set converges to a stable set. Therefore, the perturbations about the mappings and the constraint set in this paper are reasonable. The study of the convergence of solutions to perturbed systems can contribute to the idea of solving complex problems by approximating them with the solutions of a simpler series of problems. Furthermore, it can provide ideas for the design of related algorithms. Thus, it would be important and interesting to consider the stability of the stochastic system ${\bf SPEP}(b,\sigma,f,G,\tilde{F};K,y_0)$.

To this end, consider the multi-parameter system ${\bf SPEP}(b_{\lambda},\sigma_{\lambda},f_{\lambda},G_{\lambda},\tilde{F}_{\lambda};K_{\mu},y_0)$ with $\lambda,\mu\in(M,d)$ and its perturbed systems ${\bf SPEP}(b_{\lambda_m},\sigma_{\lambda_m},f_{\lambda_m},
G_{\lambda_m},\tilde{F}_{\lambda_m};K_{\mu_n},y_0)$ with $\lambda_m,\mu_n\in(M,d)$. By using Theorem \ref{multi}, we can obtain the following stability result for the stochastic spatial price equilibrium problem.
\begin{thm}
  Assume $\lambda_m, \mu_n\in(M,d)$ such that $\lambda_m\rightarrow\lambda$ and $\mu_n\rightarrow\mu$.
  Suppose $b_{\lambda_m},b_{\lambda};\sigma_{\lambda_m},\sigma_{\lambda};f_{\lambda_m},f_{\lambda};$ $G_{\lambda_m},G_{\lambda};F_{\lambda_m},F_{\lambda};\tilde{F}_{\lambda_m},\tilde{F}_{\lambda_m}$ satisfy the similar conditions in Assumption \ref{ass}, and $K_{\mu_n}\overset{M}{\rightarrow}K_{\mu}$. Moreover, let $(y_{\lambda,\mu}(t),u_{\lambda,\mu}(t))\in \mathcal{L}^2_{ad}([0,T]\times \Omega,\mathbb{R}^p)\times U_{\mu}[0,T]$ and $(y_{\lambda_m,\mu_n}(t),u_{\lambda_m,\mu_n}(t))\in \mathcal{L}^2_{ad}([0,T]\times \Omega,\mathbb{R}^p)\times U_{\mu_n}[0,T]$  be the unique solution of ${\bf SPEP}(b_{\lambda},\sigma_{\lambda},f_{\lambda},G_{\lambda},\tilde{F}_{\lambda};K_{\mu},y_0)$ and ${\bf SPEP}(b_{\lambda_m},\sigma_{\lambda_m},f_{\lambda_m},G_{\lambda_m},\tilde{F}_{\lambda_m};K_{\mu_n},y_0)$,  respectively. Then $y_{\lambda_m,\mu_n}(t)\rightarrow y_{\lambda,\mu}(t)$ in $\mathcal{L}^2_{ad}([0,T]\times \Omega,\mathbb{R}^p)$ and $u_{\lambda_m,\mu_n}(t)\rightarrow u_{\lambda,\mu}(t)$ in $H[0,T]$.
  \end{thm}
\subsection{Stochastic multi-agent optimization problem}
Multi-agent optimization problems are problems in which multiple agents in a system optimise a given objective through cooperation, competition, or information sharing. This type of problem is widespread in many real-world applications \cite{Nedic2014}, including scenarios such as wireless sensor networks, intelligent transportation systems, and robot group cooperation. It is therefore particularly important to explore effective optimization strategies for multiple agents.

With the expansion of the system size and the number of agents, the complexity of the optimization problem increases. While pursuing its own goals, each agent must also take into account the behavior and decisions of other agents, and this interdependence makes the problem more difficult to solve. In addition, the uncertainty of the environment in which the agents are located also poses challenges to the optimization process. In this context, game theory, as a tool for studying the interaction and strategy choice of decision makers, provides a theoretical basis for cooperation and competition in Multi-agent systems. It has been shown that the states of Multi-agent systems can exhibit memory, jumps and uncertainty \cite{Liu2021,Chen2021,WBang2013}. Therefore, we consider here a stochastic fractional differential game problem with L\'{e}vy jump about multi-agent and solve the Nash equilibrium of the problem using the stochastic fractional differential variational inequality with L\'{e}vy jump in this subsection. In addition, Rockafellar \cite{Rockafellar2018} studied the local stability of the Nash equilibrium in the framework of multi-agent optimization problem using variational analysis in 2018, and here we study the stability of the Nash equilibrium for the stochastic multi-agent optimization problem using the results in Section 4.

Denote by $x\in\mathcal{L}^2_{ad}([0,T]\times \Omega,\mathbb{R}^p)$ and $u\in \mathcal{L}^2_{ad}([0,T]\times \Omega,\mathbb{R}^p)$ the dynamic states and strategies of agents, respectively. For each $1\leq i\leq p$, $x^i$ and $u^i$ are, respectively, the state and strategy of agent $i$. Furthermore, $x^{-i}=(x^1,x^2,\cdots,x^{i-1},x^{i+1},\cdots,x^p)$ and $u^{-i}=(u^1,u^2,\cdots,u^{i-1},u^{i+1},\cdots,u^p)$. For our discussion, we assume that $i$-th agent's strategy set is $$U^i[0,T]=\left\{u^i(t,\omega)\in\mathcal{L}^2_{ad}([0,T]\times\Omega,\mathbb{R}): u(t,\omega)\in K^i, \;a.e.\,t\in[0,T],\; a.s.\,\omega\in\Omega\right\},$$
where $K^i$ is closed convex subset of $\mathbb{R}$, and the $i$-th agent's strategy is independent of other agents' strategies. Agent $i$'s cost function $\theta^i(x,u)$ depends on all agents' states and strategies. A Nash equilibrium for the stochastic fractional differential game with L\'{e}vy jump about multi-agent is to find a pair $(\bar{x},\bar{u})$  such that for every $i$, $(\bar{x}^i,\bar{u}^i)$ is a solution of the following problem
\begin{align}\label{agent1}
&minimize_{(x^i,u^i)}\; \;\theta^i(x^i,\bar{x}^{-i},u^i,\bar{u}^{-i})\\
&subject\; to \nonumber\\
&\qquad\qquad dx^i(t)=b^i(t,x^i(t-),,u^i(t-))dt+\sigma_1(t,x^i(t-),u^i(t-))(dt)^\alpha+\sigma^i(t,x^i(t-),u^i(t-))dB(t)\nonumber\\
&\qquad\qquad \qquad\qquad\mbox{} +\int_{\| y\|<c}G^i(t,x^i(t-),u^i(t-),y)\tilde{N}(dt,dy), \; \alpha \in (\frac{1}{2},1),\nonumber\\
&\qquad\qquad x^i(0)=p^i_0,\nonumber\\
&\qquad\qquad u^i\in U^i[0,T].\nonumber
\end{align}

It is worth noting that a Nash equilibrium can be obtained by solving a stochastic fractional differential inequality with L\'{e}vy jump. To this end, we assume $\theta^i(x^i,\bar{x}^{-i},u^i,\bar{u}^{-i})$ is convex and continuously differentiable in $u^i$, $\tilde{F}(x,u)=(\triangledown_{u^1}\theta^1(x,u),\triangledown_{u^2}\theta^2(x,u),\cdots,\triangledown_{u^p}\theta^p(x,u))$, $U[0,T]=\prod_{i=1}^pU^i[0,T]$ and $K=\prod_{i=1}^pK^i$.
Moreover we define a function $F(t,\omega,x(t,\omega),u(t,\omega))=\tilde{F}(x,u)\;a.e.\,t\in[0,T],\; a.s.\,\omega\in\Omega$. Then we have the following result.
\begin{thm}\label{nash}
$(x,u)$ is a Nash equilibrium  for the
stochastic fractional differential game with L\'{e}vy jump \eqref{agent1} if and only if $(x,u)$ solves the following stochastic fractional differential variational inequality with L\'{e}vy jump
\begin{align}\label{agent2}
\begin{cases}
 dx(t)=b(t,x(t-),u(t-))dt+\sigma_1(t,x(t-),u(t-))(dt)^\alpha+\sigma(t,x(t-),u(t-))dB(t)\\
 \qquad\qquad\mbox{} +\int_{\| y\|<c}G(t,x(t-),u(t-),y)\tilde{N}(dt,dy), \; \alpha \in (\frac{1}{2},1),\\
 x(0)=p_0,\\
 \langle F(t,\omega,x(t,\omega),u(t,\omega)),v-u(t,\omega)\rangle\geq0, \; \forall v\in K, \; a.e. \, t\in [0,T], \; a.s. \, \omega\in\Omega,
 \end{cases}
\end{align}
where
\begin{align*}
\begin{cases}
b(t,x(t-),u(t-))=(b^1(t,x^1(t-),,u^1(t-)),b^2(t,x^2(t-),,u^2(t-)),\cdots,b^p(t,x^p(t-),u^p(t-))),\\
\sigma_1(t,x(t-),u(t-))=(\sigma_1^1(t,x^1(t-),,u^1(t-)),\sigma_1^2(t,x^2(t-),u^2(t-)),\cdots,\sigma_1^p(t,x^p(t-),u^p(t-))),\\
\sigma(t,x(t-),u(t-))=(\sigma^1(t,x^1(t-),,u^1(t-)),\sigma^2(t,x^2(t-),,u^2(t-)),\cdots,\sigma^p(t,x^p(t-),u^p(t-))),\\
G(t,x(t-),u(t-),y)=(G^1(t,x^1(t-),u^1(t-),y),G^2(t,x^2(t-),u^2(t-),y)\cdots,G^p(t,x^p(t-),u^p(t-),y)).
\end{cases}
\end{align*}
\end{thm}
\begin{proof}
First of all,  it follows from Lemma \ref{requi} that
$$\langle F(t,\omega,x(t,\omega),u(t,\omega)),v-u(t,\omega)\rangle\geq0, \; \forall v\in K, \; a.e. \, t\in [0,T], \; a.s. \, \omega\in\Omega$$
is equivalent to the following variational inequality
\begin{equation}\label{agent3}
\langle \tilde{F}(x,u),v'-u\rangle_{H[0,T]}\geq 0,\;\forall v' \in U[0,T].
\end{equation}
By convexity and the minimum principle, one has that $(x,u)$ is the Nash equilibrium for \eqref{agent1} if and only if for each $i$, the $x^i$ satisfies the corresponding stochastic fractional differential equation with L\'{e}vy jump, and $u^i$ satisfies the following
\begin{equation}\label{agent4}
\langle \triangledown_{u^i}\theta^i(x,u),v^i-u^i\rangle_{H[0,T]} \geq0,\; \forall v^i\in U^i[0,T].
\end{equation}
And then \eqref{agent3} can be derived from \eqref{agent4}.

Conversely, if $(x,u)$ solves \eqref{agent2}, then $x$ satisfies the corresponding stochastic fractional differential equation with L\'{e}vy jump and \eqref{agent3} is satisfied. For any fixed $i$, in \eqref{agent3}, let $v^{-i}=u^{-i}$ and $v^i$ is an arbitrary element in $U^i[0,T]$. Then one has \eqref{agent4} immediately.
\end{proof}

Now we can solve the Nash equilibrium of the  stochastic fractional differential game problem with L\'{e}vy jump about multi-agent by our results presented in Section 4.

\begin{thm}
If the functions in \eqref{agent2} satisfy the Assumption \ref{ass}, then the Nash equilibrium of the  stochastic fractional differential game problem with L\'{e}vy jump about multi-agent \eqref{agent1} exists and is unique.
\end{thm}
\begin{proof}
It follows from Theorems \ref{th1} and \ref{nash} directly.
\end{proof}

In \cite{Rockafellar2018}, Rockafellar tied the local stability of Nash
equilibrium in a game of multi-agent optimization problem to a parameterized
variational inequality. However, the author only considered the case where the functions are perturbed by one parameter, and the case where the strategy set is perturbed is often much more complex. Here we consider the case where the functions and strategy set are perturbed by two parameters. And the parameterized stochastic fractional differential variational inequality with L\'{e}vy jump called ${\bf MAS}(\lambda,\mu)$ is defined as follows:
\begin{align}\label{agent5}
\begin{cases}
 dx_{\lambda,\mu}(t)=b_{\lambda}(t,x_{\lambda,\mu}(t-),u_{\lambda,\mu}(t-))dt
 +\sigma_{1\lambda}(t,x_{\lambda,\mu}(t-),u_{\lambda,\mu}(t-))(dt)^\alpha\\
 \qquad\qquad\;\mbox{}+\sigma_{\lambda}(t,x(t-),u(t-))dB(t)+\int_{\| y\|<c}G_{\lambda}(t,x_{\lambda,\mu}(t-),u_{\lambda,\mu}(t-),y)\tilde{N}(dt,dy), \; \alpha \in (\frac{1}{2},1),\\
 x_{\lambda,\mu}(0)=p_0,\\
 \langle F_{\lambda}(t,\omega,x_{\lambda,\mu}(t,\omega),u_{\lambda,\mu}(t,\omega)),v-u_{\lambda,\mu}(t,\omega)\rangle\geq0, \; \forall v\in K_{\mu}, \; a.e. \, t\in [0,T], \; a.s. \, \omega\in\Omega.
 \end{cases}
\end{align}

By employing Theorem \ref{multi}, we have the following stability result for the Nash equilibrium of the stochastic multi-agent optimization problem.

\begin{thm}
Assume $\lambda_m, \mu_n\in(M,d)$ such that $\lambda_m\rightarrow\lambda$ and $\mu_n\rightarrow\mu$. Suppose  $b_{\lambda_m},b_{\lambda};\sigma_{\lambda_m},\sigma_{\lambda};$ $\sigma_{1\lambda_m},\sigma_{1\lambda};G_{\lambda_m}, G_{\lambda}; F_{\lambda_m}, F_{\lambda};\tilde{F}_{\lambda_m},\tilde{F}_{\lambda_m}$ satisfy Assumption \ref{ass} and $K_{\mu_n}\overset{M}{\rightarrow}K_{\mu}$. Let $(x_{\lambda,\mu}(t),u_{\lambda,\mu}(t))$ be the unique solution of system ${\bf MAS}(\lambda,\mu)$
and $(x_{\lambda_m,\mu_n}(t),u_{\lambda_m,\mu_n}(t))$ be the unique solution of system ${\bf MAS}(\lambda_m,\mu_n)$. Then $x_{\lambda_m,\mu_n}(t)\rightarrow x_{\lambda,\mu}(t)$ in $\mathcal{L}^2_{ad}([0,T]\times \Omega,\mathbb{R}^p)$ and $u_{\lambda_m,\mu_n}(t)\rightarrow u_{\lambda,\mu}(t)$ in $\mathcal{L}^2_{ad}([0,T]\times \Omega,\mathbb{R}^p)$.
\end{thm}

\section{Conclusion}

In this paper, we study the multi-parameter stability of SFDVI with L\'{e}vy jump. By using some known results concerned with the Mosco convergence and convex analysis, we show that, for each $u\in H[0,T]$, $P_{U_{\mu_n}[0,T]}(u)$ converges strongly to $P_{U_{\mu}[0,T]}(u)$ under the assumption $K_{\mu_n}\overset{M}{\rightarrow}K_{\mu}$, where $U_{\mu}[0,T]$ and $U_{\mu_n}[0,T]$ are Hilbert spaces formed by squared integrable $\mathcal{F}_t$-adapted stochastic processes  in the ranges $k_{\mu}$ and $k_{\mu_n}$, respectively. Moreover, by employing the projection methods and some inequality techniques, we prove that the sequence of solutions of the perturbed systems converges strongly to the solution of the original system. Finally, our abstract results are applied to obtain the multi-parameter stability of solutions for the stochastic spatial price equilibrium problem and also the stability of Nash equilibrium for multi-agent optimization problem.

It is worth mentioning that Zhang et al. \cite{zhang2023s} proposed the Euler scheme for solving a class of SDVIs and
applied their results to the electrical circuits with diodes and the collapse of the bridge problems in stochastic environment. However,  to the best of authors' knowledge, numerical methods for solving SFDVI with L\'{e}vy jump have not been considered in the literature. Thus, it would be important and interesting to develop some numerical methods  for solving SFDVI with L\'{e}vy jump. We leave these as our future work.

\end{document}